\documentclass[11pt,a4paper]{article}

\usepackage{graphicx}
\usepackage{amsmath}
\usepackage{amssymb}
\usepackage{enumerate}
\usepackage{amsthm}
\usepackage{mathrsfs}
\usepackage[T1]{fontenc}
\usepackage[latin1]{inputenc}
\usepackage{color}

\let\BFseries\bfseries\def\bfseries{\BFseries\mathversion{bold}} % formulas in headings bold
\newtheorem{thm}{Theorem}

\newtheorem{lemma}[thm]{Lemma}
\newtheorem{rem}[thm]{Remark}

\theoremstyle{definition}

\newcommand{\PRO}{\mathbb{P}}
\newcommand{\ER}{\mathbb{E}}
\newcommand{\eps}{\varepsilon}

\newcommand{\dd}{\text{\rm d}}

\newcommand{\mD}{{\mathcal D}}

\newcommand{\RV}{{\mathcal RV}}

\begin{document}

\title{The first passage time problem  over a moving boundary \\ \vspace*{0.4 cm}  for asymptotically stable  L\'evy processes \\}

\author{\renewcommand{\thefootnote}{\arabic{footnote}} {\sc Frank Aurzada}\footnotemark[1] \hspace{2pt} and \renewcommand{\thefootnote}{\arabic{footnote}}{\sc Tanja Kramm}\footnotemark[1]}
\bigskip \bigskip \bigskip \bigskip \bigskip \bigskip\bigskip \bigskip\bigskip \bigskip \bigskip \bigskip \bigskip \bigskip \bigskip \bigskip\bigskip \bigskip
\bigskip
\bigskip
\bigskip
\bigskip
\bigskip
\bigskip
\bigskip
\date{\today}

 \footnotetext[1]{
 Technische Universit\"at  Darmstadt, Fachbereich Mathematik, Arbeitsgruppe Stochastik, Schlossgartenstr.\ 7, 64289 Darmstadt, Germany
{\sl aurzada@mathematik.tu-darmstadt.de},
{\sl tanjakramm@gmail.com}
 }

\date{\today}

\maketitle
\begin{abstract}
We study the asymptotic tail behaviour of the first-passage time  over a moving boundary for asymptotically $\alpha$-stable L\'evy processes with $\alpha<1$.
  
Our main result states that if the left tail of the L\'evy measure is regularly varying with index $- \alpha$ and the moving boundary is equal to $1 - t^{\gamma}$ for some $\gamma<1/\alpha$, then the probability that the process stays below the moving boundary has the same asymptotic polynomial order as in the case of a constant boundary. The same is true for the increasing boundary $1 + t^{\gamma}$ with $\gamma<1/\alpha$ under the assumption of a regularly varying right tail with index $- \alpha$.
 
% These results are motivated by the conjecture posed in \cite{AKS}.

\end{abstract}
% 
% \bigskip
% \bigskip

\vfill

%Research supported by DFG Emmy Noether programme.
\newpage

\section{Introduction and main results}\label{mainresult}

This paper is concerned with the asymptotic tail behaviour of the first-passage time over a moving boundary. For a stochastic process $(X(t))_{t\geq0}$ and a function $f:\mathbb{R}_+ \rightarrow \mathbb{R}$, the so-called moving boundary, the question is to determine the asymptotic rate of the probability
\begin{align}\label{problem}
\PRO \left(  X(t) \leq f(t) , \text{ } 0 \leq t \leq T  \right), \qquad \text{as } T \rightarrow \infty.
\end{align}
If this probability is asymptotically polynomial of order $-\delta$ (e.g.\ if it is regularly varying with index $-\delta$, in which case we will write $\RV (- \delta)$), the number $\delta$ is called the \textit{survival exponent} or \textit{persistence exponent}. If the function $f$ is constant then we are in the classical framework of first passage times over a constant boundary.

This problem is a classical question which is relevant in a number of different applications, a recent overview of results is presented in \cite{AurSim} and \cite{BMS13}. 

\medskip

For L\'evy processes, the study of the first passage time distribution over a \textit{constant} boundary is a classical area of research. The results follow from fluctuation theory; e.g.\ \cite{Rog} shows that $ \PRO (X(t) \leq c, \text{ } 0 \leq t \leq T ) $ varies regularly with index $- \rho \in (-1,0)$ if and only if $X$ satisfies Spitzer's condition with $\rho \in (0,1)$, that is (cf.\ \cite{BerDon}), $\PRO (X(t) >0) \to \rho$  as  $t \to \infty$.

 Similar arguments as for L\'evy processes were already used for random walks with zero mean (see e.g.\  \cite{feller}). If the process does not necessarily satisfy Spitzer's condition, various results were obtained for a constant boundary by \cite{Bal,BerDon2,Bor1,Bor2,DenShn,Don2,SupLP}.  

However, even for Brownian motion, the question involving \textit{moving} boundaries in  (\ref{problem}) is quite non-trivial. It is studied by \cite{Novref,Uch,Gae,jenler,Sal,Nov,Novneu,AK} in different ways. The works \cite{Novref} (for the increasing boundary case) and \cite{Uch} (for the case of increasing and decreasing boundary) are the first to state an integral test for the boundary $f$, for which the survival exponent remains $1/2$. More precisely, they prove under some additional regularity assumptions that
\begin{align}\label{eqn: necsufBM}
 \int_1^{\infty} |f(t)| t^{-3/2} \dd t < \infty \Longleftrightarrow  \PRO (X(t) \leq f(t), \text{ } 0\leq t \leq T) \approx  T^{-1/2}, \text{ as } T \rightarrow \infty.
\end{align}
Here and below we use the following notation for strong and weak asymptotics. We write $f \lesssim g$ if $\limsup_{x \to \infty} f(x)/g(x) < \infty$ and $f\approx g$ if  $f \lesssim g$ and $g \lesssim f$. Furthermore, $f \sim g$ if $f(x)/g(x) \rightarrow 1$ as $x \rightarrow \infty$.

In this paper we look at L\'evy processes $(X(t))_{t\geq0}$, in particular at L\'evy processes which belong to the domain of attraction of a strictly stable L\'evy process for $\alpha \in (0,1)$, that is, there exists a deterministic function $c$ such that 
\begin{align*}
 \frac{X(t)}{c(t)} \overset{d}{\longrightarrow} Z(1), \quad \text{as } t \to \infty,
\end{align*}
% \begin{align*}
% t \Psi_X \left(\frac{\lambda}{ c(t) }\right)  \longrightarrow \Psi_Z (\lambda),  \quad \text{as } t \to \infty,  
% \end{align*}
where $Z$ is a strictly stable L\'evy process with index $\alpha \in (0,1)$ and positivity parameter $\rho \in [0,1]$.  It is well known that if such a function $c$ exists, then $c \in \RV (1/\alpha)$ (cf.\  \cite{feller}, \cite{Brei}). We will write $X \in \mD(\alpha, \rho)$ in this case.

% 
% For technical reasons, we have to assume additionally that the L\'evy measure has a density on the left or right. More precisely, we will write $X \in \mD_{c-} (\alpha, \rho)$ if $X$ is in the domain of attraction of a strictly stable L\'evy process with index $\alpha$ and positivity parameter $\rho$ and there exists a function $\ell$ slowly varying at zero such that
% \begin{align*}
%   \nu (\dd x) = |x|^{-\alpha-1} \ell (1/|x|) \dd x, \quad \text{ for }  x < 0.
% \end{align*}

\medskip

Let us state our results. We distinguish decreasing and increasing moving boundaries.

\smallskip

For \textit{decreasing boundaries} so far it was proved by  \cite{AKS} under the assumption that $X$ possesses negative jumps that
\begin{align*}
 \gamma < \frac12  \quad \Rightarrow \quad \PRO (X(t) \leq 1- t^{\gamma}, \text{ } t \leq T ) =  T^{-\rho+o(1)}, \quad \text{as } T \to \infty.
\end{align*}
Note that  \cite{AKS}  does not assume that $X \in \mD (\alpha,\rho)$. 
 Negative results (i.e.\ such that the survival exponent does change) are given in \cite{mogpec,GreNov}. Results similar to those for Brownian motion are only available under such heavy assumptions as bounded jumps or satisfying the Cram\'{e}r condition,  \cite{Nov2} or \cite{Novdis}.
 
\smallskip
Our main result here is the following:
\begin{thm}\label{thm: nonincreasing}
Let $\alpha \in (0,1)$, $\rho \in (0,1)$, and $\gamma > 0$.  Suppose that  $X \in \mD (\alpha, \rho)$ and the left tail of the L\'evy measure has a regularly varying density. If $\limsup_{t \to 0^+} \PRO ( X (t) \geq 0) <1$ then we have
\begin{align}\label{eqn: nonincreasing}
 \gamma < \frac{1}{ \alpha}  \quad \Rightarrow \quad \PRO \left( X(t)  \leq 1 - t^{\gamma}, \text{ } t \leq T  \right) = T^{-\rho +o(1)}, \quad \text{ as } T \to \infty.
\end{align}
\end{thm}
For \textit{increasing boundaries} until now it was known from combining \cite{GreNov} and \cite{AKS} that
\begin{align*}
 \gamma <  \max \{\tfrac12, \rho \} \quad \Rightarrow \quad \PRO (X(t) \leq 1+ t^{\gamma}, \text{ }  t \leq T ) = T^{-\rho+o(1)}, \quad \text{as } T \to \infty.
\end{align*}
In this case, we have a similar statement as for decreasing boundaries.
\begin{thm}\label{thm: nondecreasing}
 Let $\alpha \in (0,1)$, $\rho \in (0,1)$, and $\gamma > 0$.  Suppose that  $X \in \mD (\alpha, \rho)$ and that the right tail of the L\'evy measure has a regularly varying density. Then we have
\begin{align}\label{eqn: nondecreasing}
 \gamma < \frac{1}{ \alpha}   \quad \Rightarrow \quad \PRO \left( X(t)  \leq 1 + t^{\gamma}, \text{ } t \leq T  \right) = T^{-\rho +o(1)}, \quad \text{ as } T \to \infty.
\end{align}
\end{thm}

Let us give a few comments on these results, in particular on the conditions on the L\'evy process.

First, note that  $\frac{1}{\alpha}> 1 \geq \max \{\frac12 ,\rho\}$  for $X \in \mD (\alpha, \rho)$ 
%(cf.\ \cite{Zolo}) 
and thus,  Theorem \ref{thm: nonincreasing} improves the results of \cite{AKS} and Theorem \ref{thm: nondecreasing} improves the results of \cite{AKS} and  \cite{GreNov}. Note that \cite{GreNov} determines exact asymptotics so that  \cite{GreNov} gives a more precise result for $\gamma < \rho$. Nevertheless, our approach provides a larger class of functions where $\rho$ remains the value of the survival exponent; and this was the main motivation of this work. Furthermore, our results indicate that the class of boundaries where the survival exponent remains the same as for the constant boundary case also depends on the tail of the L\'evy measure and not  only on $\rho$ in contrast to what the results of \cite{GreNov} seem to suggest.

Intuitively, a L\'evy process with a higher fluctuation (i.e.\ a smaller index $\alpha$) follows a boundary easier and thus allows a larger class of boundaries where the survival exponent does not change compared to the constant boundary case. This intuition gives the main idea of the proofs.

\medskip

Second, note that the assumption $\rho\in(0,1)$ excludes the case where the stable process $Z$ is spectrally positive (resp.\ negative) with index $\alpha$. That means we assume a regularly varying left tail for the decreasing boundary and a regularly varying right tail for the increasing boundary. The regularly varying left (resp.\ right) tail with index $-\alpha$ of the L\'evy measure of $X$ is an important property to show Theorem \ref{thm: nonincreasing} (resp.\ Theorem \ref{thm: nondecreasing}). Without these assumptions our approach does not work. % Note that in the spectrally negative case $\alpha \rho =1$  holds; and the relation (\ref{eqn: nondecreasing}), the increasing situation, was shown by \cite{GreNov} for $\gamma < 1/\alpha$, even providing the exact strong asymptotics.  

Essentially, our proof is based on transforming the moving boundary problem to the constant boundary case. For this purpose,  the regularly varying left (resp.\ right) tail is used. Hence, our proof gives some hope to be  generalized to other L\'evy processes such as processes indicated in \cite{DenShn}.

\begin{rem}\label{rem: spitzer}
 Since the process $X$ in Theorem \ref{thm: nonincreasing} (resp.\ Theorem \ref{thm: nondecreasing}) satisfies Spitzer's condition with parameter $\rho \in (0,1)$,  \cite{Rog} (or \cite{Bin})  yields
\begin{align}\label{eqn: spitzer}
 \PRO \left( X(t)  \leq 1 , \text{ } t \leq T  \right) = T^{-\rho + o(1)}, \quad \text{ as } T \to \infty.
\end{align}
This property provides immediately the upper (resp.\ lower) bound for (\ref{eqn: nonincreasing}) (resp.\ (\ref{eqn: nondecreasing})). Furthermore, note that for $\rho =1$ the equation (\ref{eqn: spitzer}) is not necessarily true (cf. \cite{DenShn}). 
\end{rem}

\begin{rem}
 An important idea for our proofs is to use a modified version of Theorem 3.1 in \cite{SupLP} to show the lower (resp.\ upper) bound of  (\ref{eqn: nonincreasing}) (resp.\  (\ref{eqn: nondecreasing})). For this purpose, the assumption  $\limsup_{t \to 0^+} \PRO ( X (t) \geq 0) <1$ in Theorem \ref{thm: nonincreasing} is required (see Corollary 3.4 in  \cite{SupLP} for more details). However, we believe that this may be of technical matter.

\end{rem}

% \begin{rem}
%  Theorems  \ref{thm: nonincreasing}  and \ref{thm: nondecreasing} are also true for $\alpha \geq 2$. However, in view of \cite{AKS} these results become redundant.
% \end{rem}

%  \begin{rem}
% An asymptotically stable L\'evy process with index $\alpha =1$ satisfies Spitzer's condition with parameter $\rho \in (0,1)$ if and only if the L\'evy measure is symmetric (cf.\ Property 1.2.8 in \cite{SamTaq}). In this case our approach does not work since in our proof a slight change is made to the skewness of the L\'evy measure and thus, it is not symmetric anymore.  
% \end{rem}

 \begin{rem} The additional assumption in Theorem \ref{thm: nonincreasing} (Theorem \ref{thm: nondecreasing}, respectively) that the left (right, respectively) tail of the L\'evy measure have a density is required for technial purposes only.
\end{rem}

\begin{rem}\label{rem:deleted}
 Unfortunately, the case $\alpha\geq 1$ has to be excluded from Theorems~\ref{thm: nonincreasing} and~\ref{thm: nondecreasing}. In the proof of the theorem, we split the L\'evy process $X$ into two independent parts: a small subordinator part that is supposed to follow the boundary and a remainder that has the usual event of non-exiting a constant boundary. The method of proof to control the subordinator part does not seem to work when $\alpha\geq 1$, as then the subordinator has a constant (even thouh small) drift.
\end{rem}

Let us mention that related topics have been discussed like the moments (\cite{DonMal,Gut,Rot}), the finiteness (\cite{DonMal2}), and the stability (\cite{GrifMal}) of the first passage time. Random boundaries were studied in \cite{VO,PevShi}.

\medskip

We briefly recall the basic facts of L\'evy processes. A L\'evy process   $(X(t))_{t\geq0}$ possesses stationary and independent increments, c\'adl\'ag  paths, and $X(0 )= 0$  (see \cite{bertoin}, \cite{sato}). By the L\'evy-Khintchine formula, the characteristic function of a marginal of a L\'evy process $(X(t))_{t\geq 0}$ is given by
$  \ln \ER ( e^{iuX(t)} ) = t \Psi(u)$, for every $ u \in \mathbb{R}$, where
\begin{align}\label{char}
 \Psi (u) =   i b u - \frac{\sigma^2}{2} u^2+ \int_{\mathbb{R} } (e^{iux}-1 - \mathbf{1}_{\{ |x| \leq 1  \}} iux) \nu ( \dd x),
\end{align}
for parameters  $\sigma^2 \geq 0 $,  $b \in \mathbb{R}$, and a positive measure $\nu $ concentrated on $\mathbb{R}\backslash \{0\}$, called L\'evy measure, satisfying 
\begin{align*}
\int_{\mathbb{R}} (1 \wedge x^2) \nu ( \dd x) < \infty.
\end{align*}
 For a given triplet $(b, \sigma^2, \nu)$ there exists a L\'evy process $(X(t))_{t\geq 0}$ such that (\ref{char}) holds, and its distribution is uniquely determined by its triplet. For the tails of the Le\'vy measure we will write   
\begin{align*}
\nu_+ (x) = \nu \Big(  (x, \infty ) \Big) \quad  \text{ and } \quad \nu_- (x) = \nu \Big( (-\infty, -x) \Big) . 
\end{align*}

\medskip
 The proof of Theorem \ref{thm: nonincreasing}, the case of negative  boundaries, is given in Section \ref{proofthmnoninc}, whereas Section \ref{proofthmnondec} contains the proof for positive boundaries, Theorem \ref{thm: nondecreasing}. For reasons of clarity and readability some parts of the proof are separated in Section \ref{helpresult} and may be of independent interest.

{\bf Remark:} We mention the recent preprint by Denisov and Wachtel \cite{denisovwachtel}, who obtain similar results for random walks by completely different techniques.

% \begin{rem}
%  In both Theorems, the regularity conditions on the function $f$ are for technical purposes. Trivially, both Theorems are also valid for a less regular function $g$ if there is a function $f$ satisfying the conditions in Theorem \ref{thm: nonincreasing} (Theorem \ref{thm: nondecreasing} , respectively) such that $g(s) \leq f(s) $, for all $s \geq 0$.
% \end{rem}
%%%%%%%%%%%%%%%%%%%%%%%%%%%%%%%%%%%%%%%%%%%%%%%%%%%%%%%%%%%%%%%%%%%%%%%%%%%%%%%%%%%%%%%%%%%%%%%%%%%%%%%%%%%%%%%%%%%%%%%%%%%%%%%%%%%%%%%%%%%%%%%%%%%%%%%%%%%%%%%%%%%%%%%%%%%%%%%%%%%%%%%%%%%%%%%%%%%%%%%%%%%%%%%%%%%%%%%%%%%%%%%%%%%%%%%%%%%%%%%%%%%%%%%%%%%%%%%%%%%%%%%%%%%%%%%%%%%%%%%%%%%%%%%%%%%%%%%%%%%%%%%%%%%%%%%%%%%%%%%%%%%%%%%%%%%%%%%%%%%%%%%%%
\section{Proof of Theorem \ref{thm: nonincreasing}}\label{proofthmnoninc}
 Note that the upper bound is trivial due to Remark \ref{rem: spitzer}. Hence, the following proof is devoted to the lower bound.
 
 Let $X$ be a L\'evy process with L\'evy triplet $(b, \sigma^2, \nu)$. By assumption $\nu_- \in \RV (- \alpha)$. That means there exists a function $\ell$ slowly varying at zero such that
\begin{align*}
  \nu (\dd x) = |x|^{-\alpha-1} \ell (1/|x|) \dd x, \quad \text{ for }  x < 0.
\end{align*}
% By assumption the L\'evy process $X$ with L\'evy measure $\nu_X$ belongs to the domain of attraction of a strictly $\alpha \in (0,1) \cup (1,2)$ stable L\'evy process $Z$ which possesses the following L\'evy measure 
%\begin{align}\label{def: X}
 %\nu_{Z} (d x) = \left\{ \begin{array}{ll}
  %            c_1 x^{-1-\alpha} \dd x           &\text{for } x>0,   \\
 %c_2 |x|^{-1-\alpha} \dd x    &\text{for } x<0,\\
 %                      \end{array} \right.
%\end{align}
%where $c_1\geq 0$ and $ c_2 > 0$ since $\beta = \frac{c_1- c_2 }{c_1+c_2 } \in [-1,1)$ (cf. \cite{Riv}, Lemma 5). The  parameter $\rho$ of $X$ (and $Z$) from Spitzer's condition is equal to (cf. \cite{Zolo})
%\begin{align*}
 %\rho =  \frac12 + \frac{1}{\pi \alpha} \arctan (\beta \tan (\pi \alpha/2)).
%\end{align*}
The main idea of the proof is to consider instead of the process $X$ the following two independent L\'evy processes:  Let $0\leq  \delta (T) := (\ln \ln T)^{-1} \wedge \tfrac12 \searrow 0$, for $T \rightarrow \infty$. Then, let  $ Z_T$ and $Y_T$ be L\'evy processes with characteristic exponents
\begin{align} \label{eqn: charST}
\Psi_{ Z_T} (u) :=  \int_{-\infty}^{-1} (e^{iux} -1)  \nu_{Z_T} ( \dd x), \quad u \in \mathbb{R},
\end{align}
with
\begin{align*}
  \nu_{Z_T} (\dd x) := \left\{ \begin{array}{rr} 0 , & x \geq - 1, \\
                         \delta(T) \frac{\ell (\delta (T)^{1/\alpha} /|x|) }{\ell(1/|x|)} \nu(\dd x), & x < - 1 ,
                        \end{array} \right.
\end{align*}
and
\begin{align}\label{eqn: charYT}
 \Psi_{Y_T} (u) := ibu - \frac{\sigma^2}{2}  u^2 + \int_{\mathbb{R} } (e^{iux}-1 - \mathbf{1}_{\{ |x| \leq 1  \}} iux) \nu_T (\dd x), \quad u \in \mathbb{R},
\end{align}
with 
\begin{align*}
  \nu_T (\dd x) := \left\{ \begin{array}{rr} \nu(\dd x), & x \geq - 1, \\
                         \left(1-\delta(T)\frac{\ell (\delta (T)^{1/\alpha} /|x|) }{\ell(1/|x|)} \right)\nu(\dd x), & x < - 1 .
                        \end{array} \right.
\end{align*}
Note that both  $\nu_T$  and $\nu_{Z_T}$  are L\'evy measures for every fixed $T>1$. We also denote $S_T:=-Z_T$.

% with L\'evy measure  $\nu_{S_T} (d x) =  \mathbf{1}_{\{x<-1\}} \delta(T) \nu_X (dx) $ and let $Y_T$ be a L\'evy process with L\'evy measure $(0,0,\nu_{Y_T})$ and L\'evy measure 
% \begin{align*}
%  \nu_{Y_T} (d x) = \left\{ \begin{array}{ll}
%               \nu_X (dx)           &\text{for } x>0, \\
%   (1-\mathbf{1}_{\{x<-1\}} \delta(T)) \nu_X (dx)    &\text{for } x<0.\\
%                        \end{array} \right.
% \end{align*}
Then,  $X = Y_T   - S_T $ for every fixed $T>1$ and $S_T$ is a subordinator for fixed $T>1$ by construction.
%  since $\nu_{S_T} (-\infty, 0) = 0 $ and 
% \begin{align*}
%  \int_0^\infty (1 \wedge x) \nu_{S_T} (\dd x) = \delta (T) \int_{-\infty}^{-1} (1 \wedge |x|) \nu (\dd x) =  \delta (T) \int_{-\infty}^{-1}  \nu (\dd x) < \infty.
% \end{align*}
 %Due to \cite{feller}, Section XVII.5 (or \cite{Brei}, Proposition 9.39; \cite{Riv}, Lemma 5) the L\'evy process $Y_T \in \mathcal{D} (\alpha, \beta (T))$ with $\alpha \in (0,1)  \cup (1,2)$ and $\beta (T) = \frac{c_1- c_2 + c_2 \delta(T) }{c_1+c_2 - c_2 \delta (T)} \in [-1,1)$. Observe that $Y_T$ satisfies Spitzer's condition with parameter (cf. \cite{Zolo})
%\begin{align*}
% \rho_Y (T) = \frac12 + \frac{1}{\pi \alpha} \arctan (\beta (T) \tan (\pi \alpha/2)).
%\end{align*}

For $\lambda >0 $  sufficiently small Karamata's Theorem (see Theorem 1.5.11 in \cite{binbook}) implies  the following estimate for the Laplace exponent of $S_T$
\begin{align}\label{eqn: abslapace}
   \ER \left( e^{ - \lambda S_T (1) } \right)  &= \exp \left( - \delta (T) \int_{-\infty}^{-1}  \left( 1 - e^{- \lambda |x|} \right) \frac{\ell (\delta (T)^{1/\alpha} /|x|) }{\ell(1/|x|)} \nu (\dd x)   \right) \notag \\
 &\leq  \exp \left( - \delta (T) \int_{-\infty}^{-1/\lambda}  \left( 1 - e^{- \lambda |x|} \right) \frac{\ell (\delta (T)^{1/\alpha} /|x|) }{\ell(1/|x|)}  \nu (\dd x)   \right) \notag \\
&\leq  \exp \left( - \frac12 \delta (T) \int_{-\infty}^{-1/\lambda}  \frac{\ell (\delta (T)^{1/\alpha} /|x|) }{\ell(1/|x|)} \nu (\dd x)   \right)  \notag \\
& \leq \exp \left( - \frac{1}{4 \alpha} \delta (T) \cdot \lambda^\alpha \ell (\lambda \delta (T)^{1/ \alpha} ) \right).
\end{align}
We record this formula here since it will be need very often. We stress that the bound is uniform in $T$. The particularly involved choice of the argument in the function $\ell$ will become clear in (\ref{eqn:2step}).

The decisive idea of our proof is the following observation: Due to the independence of $S_T$  and $Y_T$ we obtain that
\begin{align*}
\PRO &\left( X(t) \leq 1-t^\gamma, \text{ } t \leq T  \right)  = \PRO \left( Y_T (t) - S_T (t) \leq 1-t^\gamma, \text{ } t \leq T  \right) \notag \\
 & \geq \PRO \left( Y_T (t)  \leq \frac12 , \text{ } t \leq T  \right) \cdot  \PRO \left( - S_T (t)  \leq \frac12 - t^\gamma , \text{ } t \leq T  \right).
\end{align*}
The theorem is proved by applying the following two lemmas.
\begin{lemma}\label{lem: constant}
 Let $T>1$ and $\alpha \in (0,1)$.  Furthermore,  let $Y_T$ be the L\'evy process  defined in (\ref{eqn: charYT}). Then, it holds, for fixed $x >0$, that 
\begin{align*}
  \PRO \left( Y_T (t)  \leq x , \text{ } t \leq T  \right) = T^{-\rho+o(1)}, \quad \text{as } T \rightarrow \infty.
\end{align*}
\end{lemma}

%We remark that Lemma~\ref{lem: constant} also works for $\alpha\geq 1$ and asymptotically stable L\'evy processes without centering function as in (\ref{eqn:asymptoticallystable}).

% \begin{align}\label{eqn: goalY}
%  \PRO \left( Y_T (t)  \leq \frac12 , \text{ } t \leq T  \right) = T^{-\rho+o(1)}, \quad \text{as } T \rightarrow \infty
% \end{align}
\begin{lemma}\label{lem: alpha01}
 Let $N>1$ and $\alpha \in (0,1)$.  Furthermore, let $S_N$ be a subordinator whose Laplace transform satisfies (\ref{eqn: abslapace}) for  $\lambda >0$ sufficiently small. Then, it holds, for $0< \gamma <1/\alpha$,  that 
\begin{align*}
   \PRO \left( S_N (t)  \geq - \frac12 + t^\gamma , \text{ } 0 \leq  t \leq N  \right) =  N^{o(1)}, \quad \text{as } N \rightarrow \infty.
\end{align*}
\end{lemma}

% We remark that Lemma~\ref{lem: alpha01} only works for $\alpha<1$.

% \begin{align}\label{eqn: goalS}
%  \PRO \left( S_T (t)  \leq - \frac12 + t^\gamma , \text{ } t \leq T  \right) = T^{o(1)}, \quad \text{as } T \rightarrow \infty.
% \end{align}
% Equation (\ref{eqn: goalY}) follows from Lemma \ref{lem: constant} below; and equation (\ref{eqn: goalS}) follows from Lemma \ref{lem: alpha01} below applied to $- S_{ T}$ using inequality (\ref{eqn: abslapace}).
 Lemmas \ref{lem: constant} and \ref{lem: alpha01} are  proved in Section \ref{helpresult} using, among others, inequality (\ref{eqn: abslapace}).

%%%%%%%%%%%%%%%%%%%%%%%%%%%%%%%%%%%%%%%%%%%%%%%%%%%%%%%%%%%%%%%%%%%%%%%%%%%%%%%%%%%%%%%%%%%%%%%%%%%%%%%%%%%%%%%%%%%%%%%%%%%%%%%%%%%%%%%%%%%%%%%%%%%%%%%%%%%%%%%%%%%%%%%%%%%%%%%%
%%%%%%%%%%%%%%%%%%%%%%%%%%%%%%%%%%%%%%%%%%%%%%%%%%%%%%%%%%%%%%%%%%%%%%%%%%%%%%%%%%%%%%%%%%%%%%%%%%%%%%%%%%%%%%%%%%%%%%%%%%%%%%%%%%%%%%%%%%%%%%%%%%%%%%%%%%%

\section{Proof of Theorem \ref{thm: nondecreasing}}\label{proofthmnondec}
Contrary to the proof of Theorem \ref{thm: nonincreasing}  the lower bound is trivial due to Remark \ref{rem: spitzer}. Hence, the following proof is devoted to the upper bound, where the idea of the proof is essentially the same as for the lower bound of Theorem \ref{thm: nonincreasing}. 

Let $X$ be a L\'evy process with L\'evy triplet $(b, \sigma^2, \nu)$. By assumption we have
\begin{align*}
 \nu (\dd x) =  x^{-\alpha-1} \ell (1/x) \dd x, \quad x > 0 
\end{align*}
where $\ell$ is a slowly varying at zero function. 
 %By assumption the L\'evy process $X$ with L\'evy measure $\nu_X$ belongs to the domain of attraction of a strictly $\alpha \in (0,1) \cup (1,2)$ stable L\'evy process $Z$ which possesses the following L\'evy measure 
%\begin{align}\label{def: Xnondecr}
% \nu_{Z} (d x) = \left\{ \begin{array}{ll}
 %             c_1 x^{-1-\alpha} dx           &\text{for } x>0,   \\
 %c_2 |x|^{-1-\alpha} dx    &\text{for } x<0,\\
  %                     \end{array} \right.
%\end{align}
%where $c_1 >0$ and $ c_2 \geq  0$ since $\beta = \frac{c_1- c_2 }{c_1+c_2 } \in (-1,1]$  (cf. \cite{Riv}, Lemma 5). The parameter $\rho$ of $X$ (and $Z$) from Spitzer's condition is equal to (cf. \cite{Zolo})
%\begin{align*}
 %\rho = \frac12 + \frac{1}{\pi \alpha} \arctan (\beta \tan (\pi \alpha/2)).
%\end{align*}
Again, the main idea of the proof is to consider instead of the process $X$ the following two independent L\'evy processes:  Let $0\leq  \delta (T) := (\ln \ln T)^{-1} \wedge \tfrac12 \searrow 0$, for $T \rightarrow \infty$. Then,  let  $S_T$ and $Y_T$ be L\'evy processes with characteristic exponents
\begin{align}  \label{eqn: charST1}
\Psi_{S_T} (u):= \int_{1}^{\infty} (e^{iux} -1)  \nu_{S_T} ( \dd x), \quad u \in \mathbb{R},
\end{align}
with
\begin{align*}
 \nu_{S_T} (\dd x ) := \left\{ \begin{array}{rr} 0, & x \leq  1, \\
                         \delta(T) \frac{\ell (\delta (T)^{1/\alpha} /x) }{\ell(1/x)}  \nu(\dd x), & x > 1 ,
                        \end{array} \right.
\end{align*}
and
\begin{align}\label{eqn: charYT1}
\Psi_{Y_T} (u) := ibu - \frac{\sigma^2}{2}  u^2 + \int_{\mathbb{R} } (e^{iux}-1 - \mathbf{1}_{\{ |x| \leq 1  \}} iux) \nu_T (\dd x), \quad u \in \mathbb{R},
\end{align}
with 
\begin{align*}
  \nu_T (\dd x) := \left\{ \begin{array}{rr} \nu(\dd x), & x \leq  1, \\
                         (1-\delta(T) \frac{\ell (\delta (T)^{1/\alpha} /x) }{\ell(1/x)}  )\nu(\dd x), & x > 1 .
                        \end{array} \right.
\end{align*}
Note that both $\nu_T$ and $ \nu_{S_T}$ are L\'evy measures for every fixed $T>1$.
% Then, let $S_T$ be a subordinator with  L\'evy measure  $\nu_{S_T} (d x) = \mathbf{1}_{\{x>1\}} \delta(T) \nu_X( dx)$ and let $Y_T$ be a L\'evy process with L\'evy measure $(0,0,\nu_{Y_T})$ and L\'evy measure 
% \begin{align*}
%  \nu_{Y_T} (d x) = \left\{ \begin{array}{ll}
%                (1-\mathbf{1}_{\{x>1\}} \delta(T))  \nu_X( dx)           &\text{for } x>0, \\
%  \nu_X ( dx)    &\text{for } x<0.\\
%                        \end{array} \right.
% \end{align*}
Thus, $X = Y_T  +S_T $ for every $T>0$. The process $S_T$ is a subordinator by construction.
% since $\nu_{S_T} (-\infty, 0) = 0 $ and by assumption on $X$
% \begin{align*}
%  \int_0^\infty (1 \wedge x) \nu_{S_T} (\dd x) = \delta (T) \int_1^\infty (1 \wedge x) \nu (\dd x) < \infty.
% \end{align*}
% Due to \cite{feller}, Section XVII.5 (or \cite{Brei}, Proposition 9.39; \cite{Riv}, Lemma 5)  the L\'evy process $Y_T \in \mathcal{D}  (\alpha, \beta (T))$ with $\alpha \in (0,1) \cup (1,2)$ and $\beta (T) = \frac{c_1- c_2 - c_1 \delta(T) }{c_1+c_2 - c_1 \delta (T)}$. Observe that $Y_T$ satisfies Spitzer's condition  with parameter (cf. \cite{Zolo}) 
%\begin{align*}
 %\rho_Y (T) = \frac12 + \frac{1}{\pi \alpha} \arctan (\beta (T) \tan (\pi \alpha/2)).
%\end{align*} 

Since $0<\alpha \gamma <1$ there exists a constant $\eps >0$ such that $\gamma \alpha + \eps < 1$ and $\gamma \alpha - \eps >0$.
 Moreover, define $T_0 := \lfloor (\ln   T  )^{\frac{3}{1-\alpha \gamma - \eps }} \rfloor$. Then, we obtain the following estimate
\begin{align}\label{eqn: goalnondecr}
\PRO \left( X(t) \leq 1 + t^\gamma, \text{ } t \leq T  \right) & = \PRO \left( Y_T (t) + S_T (t) \leq 1 + t^\gamma, \text{ } t \leq T  \right) \notag \\
 & \leq \PRO \left( Y_T (n) + S_T (n) \leq 1 + n^\gamma, \text{ } \forall n = T_0 ,..., \lfloor T\rfloor  \right) \notag \\
 & \leq \PRO \left(\left\{ Y_T (n)  \leq 1 +n^\gamma - S_T (n) ,\text{ } \forall n = T_0 ,..., \lfloor T\rfloor \right\} \right. \notag \\
&\quad \quad \left.  \cap \left\{  n^\gamma - S_T (n) \leq 0 ,\text{ } \forall n = T_0 ,..., \lfloor T\rfloor \right\}  \right)\notag \\
&\quad +   \PRO \left( \exists n \in \left\{ T_0 ,..., \lfloor T\rfloor \right\} : S_T (n)  <  n^\gamma  \right) \notag \\
 & \leq \PRO \left( Y_T (n)  \leq 1 ,\text{ } \forall n = T_0 ,..., \lfloor T\rfloor   \right)\notag \\
&\quad +   \PRO \left( \exists n \in \left\{ T_0 ,..., \lfloor T\rfloor \right\} : S_T (n)  <  n^\gamma  \right).
\end{align}
 On the one hand, for $\lambda >0 $  sufficiently small Karamata's Theorem (see Theorem 1.5.11 in \cite{binbook}) implies  the following estimate for the Laplace exponent of $S_T$
\begin{align}\label{eqn: abslaplace1}
   \ER \left( e^{ - \lambda S_T (1) } \right)  &= \exp \left( - \delta (T) \int_1^\infty  \left( 1 - e^{- \lambda x} \right)  \frac{\ell (\delta (T)^{1/\alpha} /x) }{\ell(1/x)} \nu (\dd x)   \right)  \notag \\
&\leq \exp \left( - \delta (T) \int_{1/\lambda}^\infty  \left( 1 - e^{- \lambda x} \right)  \frac{\ell (\delta (T)^{1/\alpha} /x) }{\ell(1/x)} \nu (\dd x)   \right)  \notag \\
&\leq \exp \left( - \frac12 \delta (T) \int_{1/\lambda}^\infty  \frac{\ell (\delta (T)^{1/\alpha} /x) }{\ell(1/x)}  \nu (\dd x)   \right) \notag \\
& \leq \exp \left( -\frac{1}{4\alpha}  \delta (T) \cdot \lambda^\alpha \ell (\lambda \delta (T)^{1/\alpha} ) \right).
\end{align}
Then, Chebyshev's inequality gives, for $T>1$ sufficiently large,
 \begin{align}
  \PRO \left( \exists n \in \left\{ T_0,..., \lfloor T\rfloor \right\} : S_T (n)  <  n^\gamma  \right) &\leq \sum_{n= T_0}^{\lfloor T\rfloor}   \PRO \left(   S_T (n )  <  n^\gamma  \right) \notag \\
 &= \sum_{n= T_0}^{\lfloor T\rfloor}   \PRO \left(  e^{-n^{-\gamma}  S_T (n )}  \geq   e^{-1} \right)  \notag  \\
 &\leq  \sum_{n= T_0}^{\lfloor T\rfloor}    e^{1 -\tfrac{1}{4\alpha} n^{1-\alpha \gamma} \ell (n^{-\gamma} \delta(T)^{1/ \alpha}) \delta (T) } 
 \end{align}
 Proposition 1.3.6 in \cite{binbook}  implies $\ell(\lambda) \geq \lambda^{\eps / \gamma} $, for $\lambda>0$ sufficiently small, and thus we get for $T>1$ sufficiently large  
 \begin{align}\label{eqn:estimateST}
\PRO \left( \exists n \in \left\{ T_0,..., \lfloor T\rfloor \right\} : S_T (n)  <  n^\gamma  \right) &\leq  \sum_{n= T_0}^{\lfloor T\rfloor}     e^{1 -\tfrac{1}{4\alpha} T_0^{1-\alpha \gamma- \eps}  \delta (T)^{1+ \eps / (\alpha \gamma) } }  \notag  \\
 &\leq     e^{1+ \ln \lfloor T\rfloor - \tfrac{1}{4\alpha} (\ln \lfloor T\rfloor )^3  \delta (T)^{1+\eps/ (\alpha \gamma)  }}  \notag  \\
 &\leq      e^{-  (\ln \lfloor T\rfloor )^2 }  \notag  \\
 &\leq T^{-\rho + o(1)},
 \end{align}
where we used  that $\alpha \gamma - \eps >0$ in the second last step.

% The proof is complete as soon as we know that 
% \begin{align}\label{eqn: goalYnondecr}
% \PRO \left( Y_T (n)  \leq 1 ,\text{ } \forall n =  T_0 ,..., \lfloor T\rfloor   \right) = T^{-\rho+o(1)}, \quad \text{as } T \rightarrow \infty
% \end{align}
% and on the other hand that 
% \begin{align}\label{eqn: goalSnondecr}
%  \PRO \left( \exists n \in \left\{ T_0 ,..., \lfloor T\rfloor \right\} : S_T (n)  <  n^\gamma  \right) =  T^{-\rho+o(1)}, \quad \text{as } T \rightarrow \infty.
% \end{align}
%%%%%%%%%%%%%%%%%%%%%%%%%%%%%%%%%%%%%%%%%%%%%%%%%%%%%%%%%%%%%%%%%%%%%%%%%%%%%%%%%%%%%%%%%%%%%%%%%%%%%%%%%%%%%%%%%%%%%%%%%%%%%%%%%%%%%%%%%%%%%%%%%%%%%%%%%%%%%%%%%%%%%%%%%%%%%%%%%%%%%%%%%%%%%%%%%%%%%%%%%%%%%%%%%%%%%%%%%%%%%%%%%%%%%%%%%%%%%%%%%%%%%%%%%%%%%%%%%%%%%%%%%%%%%%%%%%%%%%%%%%%%%%%%%%%%%%%%%%%%%%%%%%%%%%%%%%%%%%%
On the other hand, using the fact that the process $Y_T$ is associated (cf.\ \cite{ass} or \cite{AKS}, Lemma 14) implies
\begin{align*}
 \PRO  \left( Y_T (n)  \leq 1 ,\text{ } \forall n = T_0 ,..., \lfloor T\rfloor   \right)  &\leq \frac{\PRO \left( Y_T (n)  \leq 1 ,\text{ } \forall n =  0 ,..., \lfloor T\rfloor  \right) }{\PRO \left( Y_T (n)  \leq 1 ,\text{ } \forall n = 0,..., T_0    \right) }.
\end{align*}
Note that $S_T \geq 0$ a.s.\  since $S_T$ is a subordinator.  Hence, due to $X=Y_T+S_T$ and Remark \ref{rem: spitzer} we obtain that 
\begin{equation}\label{eqn:estimateYT}
\PRO \left( Y_T (n)  \leq 1 ,\text{ } \forall n = 0,..., T_0    \right) \geq  \PRO \left( X (n)   \leq 1 , \text{ } \forall n = 0,..., T_0   \right) =T_0^{-\rho+o(1)} .
\end{equation}
% \begin{align}\label{eqn:estimateYT}
% \PRO & \left( Y_T (n)  \leq 1 ,\text{ } \forall n = 0,..., T_0    \right) =  \PRO \left( X (n)  - S_T (n)  \leq 1 ,\text{ } \forall n = 0,..., T_0    \right) \notag \\
% & \geq  \PRO \left(\{ X (n)   \leq 1,   \text{ } \forall n = 0,..., T_0 \} \cap \{ S_T (n)\geq 0 ,  \text{ } \forall n = 0,..., T_0  \} \right) \notag \\
% %& =  \PRO \left( X (n)   \leq 1 , \text{ } \forall n = 0,..., T_0   \right) \notag \\
% & =  \PRO \left( X (n)   \leq 1 , \text{ } \forall n = 0,..., T_0   \right) =T_0^{-\rho+o(1)} .
% \end{align}
Now, setting (\ref{eqn:estimateYT}) and  (\ref{eqn:estimateST})  in  (\ref{eqn: goalnondecr}) gives
\begin{align*}
 \PRO  \left( Y_T (n)  \leq 1 ,\text{ } \forall n =  T_0 ,..., \lfloor T\rfloor   \right) & \leq  T^{- \rho + o(1)} + \PRO \left( Y_T (n)  \leq 1 ,\text{ } \forall n =  0 ,..., \lfloor T\rfloor  \right) \cdot T^{o(1)} 
\end{align*}
The theorem is proved by applying the following lemma:
\begin{lemma}\label{lem: constant1}
 Let $T>1$ and $\alpha \in (0,1)$. Furthermore,  let $Y_T$ be a L\'evy process defined in (\ref{eqn: charYT1}). Then, for $x >0$
\begin{align*}
  \PRO  \left( Y_T (n)  \leq x ,\text{ } \forall n =  1,..., \lfloor T \rfloor   \right)  \leq T^{-\rho+o(1)}, \quad \text{as } T \rightarrow \infty.
\end{align*}
\end{lemma}
%%%%%%%%%%%%%%%%%%%%%%%%%%%%%%%%%%%%%%%%%%%%%%%%%%%%%%%%%%%%%%%%%%%%%%%%%%%%%%%%%%%%%%%%%%%%%%%%%%%%%%%%%%%%%%%%%%%%%%%%%%%%%%%%%%%%%%%%%%%%%%%%%%%%%%%%%%%%%%%%%%%%%%%%%%%%%%%%%%%%%%%%%%%%%%%%%%%%%%%%%%%%%%%%%%%%%%%%%%%%%%%%%%%%%%%%%%%%%%%%%%%%%%%%%%%%%%%%%%%%%%%%%%%%%%%%%%%%%%%%%%%%%%%%%%%

%%%%%%%%%%%%%%%%%%%%%%%%%%%%%%%%%%%%%%%%%%%%%%%%%%%%%%%%%%%%%%%%%%%%%%%%%%%%%%%%%%%%%%%%%%%%%%%%%%%%%%%%%%%%%%%%%%%%%%%%%%%%%%%%%%%%%%%%%%%%%%%%%%%%%%%%%%%%%%%%%%%%%%%%%%%%%%%%%%%%%%%%%%%%%%%%%%%%%%%%%%%%%%%%%%%%%%%%%%%%%%%%%%%%%%%%%%%%%%%%%%%%%%%%%%%%%%%%%%%%%%%%%%%%%%%%%%%%%%%%%%%%%%%%%%%%%%%%%%%%%%%%%%%%%%%%%%%%%%%%%%%%%%%%%%%%%%5
\section{First passage times of a time-dependent L\'evy process}\label{helpresult}
In the first section, we briefly recall the basic notations of fluctuation theory for L\'evy processes. Furthermore, we give some properties of  $Y_T$ and $S_T$ defined in (\ref{eqn: charYT}) and  (\ref{eqn: charST}) (resp.\  (\ref{eqn: charYT1}) and  (\ref{eqn: charST1})). In the subsequent section, we prove Lemma \ref{lem: constant} and \ref{lem: constant1}.  Section \ref{sec:sub} provides the proof of Lemma \ref{lem: alpha01}.

\subsection{Preliminaries and Notations}

For the notation of this section, we refer to \cite{bertoin} or \cite{kyp}, Section 6.4. Let $L$ be the local time of a general L\'evy process $X$ reflected at its supremum $M $ and denote by $L^{-1}$ the right-continuous inverse of $L$, the inverse local time. This is a (possibly killed) subordinator, and $H(s) := X(L^{-1}(s))$ is another (possibly killed) subordinator called ascending ladder height process. The Laplace exponent  of the (possibly killed) bivariate subordinator $(L^{-1}(s),H(s))$ $(s \leq L(\infty))$ is denoted by $\kappa (a,b)$,
\begin{align}\label{eqn:laplaceexpk}
\kappa (a,b) = c \exp \left( \int_0^\infty \int_{[0,\infty)} (e^{-t} - e^{-at-bx}) t^{-1} \PRO (X (t) \in \dd x) \dd t  \right),
\end{align}
where $c$ is a normalization constant of the local time. Since our results are not effected by the choice of $c$ we assume $c=1$. 

Following \cite{bertoin}, we define the renewal function of the process $H$ by
\begin{align}\label{eqn:renewal}
V(x) = \int_0^\infty \PRO (H(s) <x) \dd s
\end{align}
and for $z \geq 0$
\begin{align*}
V^z (x) = \ER \left( \int_0^\infty e^{-zt} \mathbf{1}_{[0,x)} (M(t)) \dd L(t) \right).
\end{align*}

Until further notice, we denote by $\kappa_T$ the Laplace exponent of the inverse local time $L_T^{-1}$ of $Y_T$ defined in (\ref{eqn: charYT}) (resp.\ (\ref{eqn: charYT1})) and $V_T$ the renewal function of the ladder height process $H_T$ of $Y_T$. 
Furthermore, denote by $L^{-1}$ the inverse local time of $X$ defined in Theorem \ref{thm: nonincreasing} (resp.\ Theorem \ref{thm: nondecreasing}) and $H$ be the corresponding ladder height process. Let $\kappa$ be the Laplace exponent of $(L^{-1},H)$. 

The next lemma shows the convergence of the renewal function $V_T$ to $V$.
\begin{lemma}\label{lem: VT}
Let $T>1$. Then, for all $x>0$ that are points of continuity of $V$, we have 
\begin{align*}
\lim_{T \to \infty} V_T (x) = V(x). 
\end{align*}
\end{lemma}
\begin{proof}
 The Continuity Theorem (cf.\  \cite{feller}, Theorem XIII.1.2)  gives $Y_T (s) \overset{d}{\longrightarrow} X(s)$, as $T \to \infty$, for all $s \geq 0$.  
%The Laplace exponent $\kappa_T$ of the ladder height process $H_T$ is  for every fixed $T>0$  given by (cf. \cite{Fri74} or \cite{bertoin}, Corollary VI.10) for all $s \geq 0$
%\begin{align*}
 % \ER \left( e^{- \lambda H_T (s)} \right) = e^{_- s \psi_T (\lambda) }= e^{ - s b \exp \left( \int_0^\infty \int_0^\infty t^{-1} \left( e^{-t} - e^{-\lambda x}\right) \PRO ( Y_T (t) \in \dd x) \dd t \right)}, \quad \lambda >0, 
%\end{align*}
%where $b>0$ is independent of $T$.
 %Since $Y_T (s) \overset{d}{\longrightarrow} X(s)$, as $T \to \infty$, 
Since $e^{-\lambda x}$ is bounded for all $x,\lambda \geq 0$,  Theorem VIII.1.1 in \cite{feller} implies that 
$$
e^ {-s \kappa_T(0,\lambda)} = \ER  \left( e^{- \lambda H_T (s)} \right) \longrightarrow  \ER \left( e^{- \lambda H (s)} \right)=  e^{-s \kappa (0,\lambda)}.
$$
Since the Laplace transform of the measure $V_T(\dd x)$ is $\kappa_T(0,\lambda)$:
$$
\int_{[0,\infty)} e^{-\lambda x} V_T(\dd x) = \int_0^\infty \int_{[0,\infty)} e^{-\lambda x} \PRO_{H_T(s)}( \dd x) \dd s = \int_0^\infty e^{- s \kappa_T(0,\lambda)} \dd s = \kappa_T(0,\lambda)^{-1},
$$
and the Laplace transform of $V(\dd x)$ is $\kappa(0,\lambda)$, respectively, this already shows the assertion.
\end{proof}

The next lemma characterises the tail behaviour of $S_T$ defined in (\ref{eqn: charST1}).
\begin{lemma}\label{lem:ST} Let $\alpha<1$.  Let $T>1$ and $S_T$ be the subordinator defined in (\ref{eqn: charST1}). Let $c$ be the norming sequence of $X$. Then,  there exists a constant $C>0$ such that for all integer $t > \delta (T)^{-1/2}$ and $T$ sufficiently large
\begin{align}\label{eqn:ST}
 \PRO ( S_T (t) > c(t) \delta (T)^{\frac{1}{2 \alpha}} ) \leq C \delta (T)^{\frac13}.
\end{align}
\end{lemma}
\begin{proof}
The idea of the proof is to apply a large deviation principle. For this purpose, let $(S^*_T(t))_{t\geq 0}$ be a subordinator independent of $S_T$ with exponent

$$
\Psi_{S^*_T} (u):= \int_{\delta (T)^{1/\alpha} }^{1} (e^{iux} -1)   \nu_{\tilde{S}_T} ( \dd x), \quad u \in \mathbb{R},
$$

and L\'evy measure

\begin{align*}
 \nu_{S^*_T} (\dd x ) := \left\{ \begin{array}{rr} \delta (T) x^{-\alpha-1} \ell ( \delta (T)^{1/\alpha} /x) \dd x, & \delta (T)^{1/\alpha} \leq x \leq 1, \\
                        0, & \text{otherwise.}
                        \end{array} \right.
\end{align*}

Further, define the following L\'evy processes by their exponents:

\begin{align*}
% \Psi_{\tilde{Z}} (u) &:= \int_{1}^{\infty} (e^{iux} -1) x^{-1 -\alpha } \dd x, \quad u \in \mathbb{R},\\
\Psi_{\tilde{X}} (u)&:= \int_{1}^{\infty} (e^{iux} -1) \nu(\dd x), \quad u \in \mathbb{R},\\
\Psi_{\tilde{S}_T} (u)&:= \int_{\delta (T)^{1/\alpha} }^{\infty} (e^{iux} -1)   \nu_{\tilde{S}_T} ( \dd x), \quad u \in \mathbb{R},
\end{align*}
with
\begin{align*}
 \nu_{\tilde{S}_T} (\dd x ) := \left\{ \begin{array}{rr} \delta (T) x^{-\alpha-1} \ell ( \delta (T)^{1/\alpha} /x) \dd x, & x \geq  \delta (T)^{1/\alpha}, \\
                        0, & x < \delta (T)^{1/\alpha} .
                        \end{array} \right.
\end{align*}

\textit{1st.\ Step:} It is clear from the construction that $\tilde S_T\overset{d}{=}S^*_T + S_T$. So that from the fact that $S^*_T>0$ we obtain
\begin{align}\label{eqn:1step}
 \PRO ( S_T (t) > \lambda  ) & \leq \PRO ( \tilde{S}_T (t) > \lambda ), \quad \text{ for every } \lambda,t > 0,
\end{align}
Further, by construction,
\begin{align}\label{eqn:2step}
  (\frac{\tilde{S_T} (t) }{\delta (T)^{1/\alpha} })_{t\geq 0}  \overset{d}{=} (\tilde{X} (t))_{t\geq 0}.
 \end{align}

% By construction we have for every $\lambda > 0$
% \begin{align*}
%   \nu_{S_T} ( x \in \mathbb{R} : x > \lambda) \leq  \nu_{\tilde{S}_T} ( x \in \mathbb{R} : x > \lambda) .
% \end{align*}
% Then, it follows from \cite{SamTaq94} that  (\ref{eqn:1step}) holds. Let us mention that (\ref{eqn:1step}) is an extention of Slepian's inequality for L\'evy processes. \\
% \textit{2nd.\ Step:} Now, we will prove that for all $T>1$
% \begin{align}\label{eqn:2step}
%  \frac{\tilde{S_T} (t) }{\delta (T)^{1/\alpha} }  \overset{d}{=} \tilde{X} (t), \text{ for every } t \geq 0.
% \end{align}
% Integration by substitution gives for every $\lambda > 0$  
% \begin{align*}
%  t \Psi_{\tilde{S}_T} \left( \frac{\lambda }{\delta (T)^{1/\alpha} c(t)} \right) &= t \exp \left( \int_{\delta (T)^{1/\alpha} }^{\infty} \left( e^{i\left( \frac{\lambda }{\delta (T)^{1/\alpha} c(t)} \right) x} -1 \right)   \nu_{S_T} ( \dd x) \right) \\
% & = t \exp \left( \int_{1 }^{\infty} \left( e^{i\left( \frac{\lambda }{c(t)} \right) x} -1 \right)  \delta (T)^{1/\alpha}    \nu_{S_T} ( \dd x  \delta (T)^{1/\alpha} ) \right) \\
% & = t \exp \left( \int_{1 }^{\infty} \left( e^{i\left( \frac{\lambda }{c(t)} \right) x} -1 \right)   \nu ( \dd x) \right) \\
% & =t  \Psi_{\tilde{X}} (\lambda/c(t)) , \quad \text{ for all } t \geq 0,
% \end{align*}
% and this proves (\ref{eqn:2step}). 
%  Next, it will be proved that
% \begin{align}\label{eqn:3step}
%   \frac{\tilde{X}(t) }{c(t)} \overset{d}{\longrightarrow} \tilde{Z} (1), \quad \text{ as } t \to  \infty. 
% \end{align}
Recall that $X$ belongs to the domain of attraction of a strictly stable L\'evy process with norming sequence $c$ and $\nu_+ \in \RV (- \alpha)$. Hence, by construction it implies that
\begin{align}\label{eqn:3step}
\PRO (\tilde{X} (1) >.) \in \RV (- \alpha).
\end{align}
\textit{2nd.\ Step:} We show (\ref{eqn:ST}).

For $\alpha<1$, we can apply a large deviation principle (see Theorem 2.1 in \cite{DDS08}) to the process $\tilde X$ (which lies in the domain of attraction of a stable process), which proves that there exists a constant $C>1$ such that for all $t > \delta(T)^{-1/2}$ and $T$ sufficiently large
\begin{align*}
 \PRO \left( \frac{ \tilde{X} (t)}{c(t) } >  \delta (T)^{- \frac{1}{2 \alpha}} \right) %& \leq 2 \PRO \left( \frac{ \tilde{X} ([t])}{c([t]) } >  \delta (T)^{- \frac{1}{2 \alpha}} \right) \\
& \leq C t    \PRO \left(  \tilde{X} (1 ) > c(t)  \delta (T)^{- \frac{1}{2 \alpha}} \right)  \leq C \delta (T)^{ \frac{1}{3 }},
\end{align*}
where we used (\ref{eqn:3step}) in the second step.
Note that the constant $C$ does not depend on $t$. Hence, combining this with (\ref{eqn:1step}) and (\ref{eqn:2step})  leads to for all $t > \delta(T)^{-1/2}$ and $T$ sufficiently large
\begin{align*}
 \PRO \left( S_T (t) > c(t) \delta (T)^{\frac{1}{2 \alpha}} \right)
&\leq \PRO \left( \frac{ \tilde{S}_T (t)}{c(t)\delta (T)^{1/\alpha} } >  \delta (T)^{- \frac{1}{2 \alpha}} \right) 
= \PRO \left( \frac{ \tilde{X} (t)}{c(t)} >  \delta (T)^{- \frac{1}{2 \alpha}} \right) 
 \leq C \delta (T)^{1/3}.
\end{align*}

\end{proof}

\begin{rem}\label{rem:ST}
Inequality (\ref{eqn:ST})  holds as well for the subordinator $S_T$ defined in (\ref{eqn: charST}). The proof is essentially the same and is omitted.
\end{rem}

\subsection{A time-dependent L\'evy process over a constant boundary}
Now, we show  Lemma \ref{lem: constant} and Lemma \ref{lem: constant1}. We calculate the asymptotic tail behaviour as $T$ tends to infinity of the first-passage time over a constant boundary for a L\'evy process which depends on the end time point $T$. Lemma \ref{lem: constant} and Lemma \ref{lem: constant1} differ only in the considered process  as well as  the time scale. 

\begin{proof}[Proof of Lemma \ref{lem: constant}]
Recall that $Y_T = X + S_T $, where $X$ is defined in Theorem \ref{thm: nonincreasing}  and $S_T$ is a subordinator defined in (\ref{eqn: charST}). Since $X \in \mD (\alpha , \rho)$ there exists a deterministic function $c: (0,\infty) \to (0,\infty)$ such that 
\begin{align*} % \label{eqn:asymptoticallystable}
 \frac{X(t)}{c(t)} \overset{d}{\longrightarrow} Z(1), \quad \text{as } t \to \infty,
\end{align*}
where $Z$ is strictly stable L\'evy process with index $\alpha \in (0,1)$ and positivity parameter $\rho \in (0,1)$.

The upper bound is trivial since $S_T \geq 0$ a.s.\ and thus
\begin{align}\label{eqn: upperbound}
 \PRO \left( Y_T (t)  \leq x , \text{ } t \leq T  \right) &=\PRO \left( X (t) + S_T(t)  \leq x , \text{ } t \leq T  \right) \notag \\
 & \leq \PRO \left( X (t)  \leq x , \text{ } t \leq T  \right) \notag \\
 & = T^{-\rho+o(1)}, \quad \text{as } T \rightarrow \infty,
\end{align}
see Remark \ref{rem: spitzer}.

The idea of the proof of the lower bound is to apply some parts of Theorem 3.1 and Corollary 3.2 in \cite{SupLP}. The main difference between our result and \cite{SupLP} is that the process $Y_T$ depends on $T$.
For this purpose, we define 
\begin{align*}
M_T (t) := \sup_{s \leq t} Y_T (s),
\end{align*}
and thus,
\begin{align*}
  \PRO \left( Y_T (t)  < x, \text{ } t \leq T  \right) = \PRO \left( M_T (T) < x \right).
\end{align*}
\textit{1st.\ Step:} 

Let $z = z(T)$ with $ T^{-1 - |o(1)| } < z < T^{-1}$.  By the estimate (3.7) in the proof of Theorem 3.1 in \cite{SupLP} we have for $M_T$:
\begin{align*}
\PRO \left( M_T (T) < x \right)& \geq \frac{\kappa_T (z,0) \PRO (M_T (1/z) \geq x) V_T (x)}{e} - z \int_0^T e^{-zt} \PRO (M_T (t) <x) \dd t. 
\end{align*}
Then, since $\PRO (M_T (t) <x)  \leq \PRO (M (t) <x) $, for all $t \geq 0$, (cf.\ (\ref{eqn: upperbound}))  the upper bound of (3.5) in \cite{SupLP} applied to  $ \PRO (M (t) <x) $ gives, for every $T>1$  and $x >0$,
\begin{align*}
\PRO \left( M_T (T) < x \right)& \geq \frac{\kappa_T (z,0) \PRO (M_T (1/z) \geq x) V_T (x)}{e} - \frac{e}{e-1} V (x) z \int_0^{T} \kappa (1/t,0) \dd t\\
& =  \frac{\kappa_T (z,0) \PRO (M_T (1/z) \geq x) V_T (x)}{e} - \frac{e}{e-1} V (x) z K(1/T) ,
\end{align*}
where 
\begin{align*}
K(s) = \int_s^\infty \frac{\kappa (z,0)}{z^2} \dd z.
\end{align*}

The assumption $\limsup_{t \to 0^+} \PRO (X(t)\geq 0) < 1$ implies that $K(s)$ is well-defined (see Corollary 3.4 in  \cite{SupLP} for more details). Since $\kappa (z,0)$ is regularly varying at zero, by Karamata's theorem (\cite{Bin}, Theorem 1.5.11) we have $K(1/T) \leq c_1(\kappa) T \kappa (1/T,0)$, for $T \geq 1$. Hence,
\begin{align*}
\PRO \left( M_T (T) < x \right)& \geq \frac{\kappa_T (z,0) \PRO (M_T (1/z) \geq x) V_T (x)}{e} - \frac{e}{e-1} V (x) z c_1(\kappa) T \kappa (1/T,0) .
\end{align*}
Lemma \ref{lem: VT} implies, for $T>1$ sufficiently large, that 
\begin{align*}
V_T (x) \geq \tfrac12  V(x).
\end{align*}
Furthermore, the inequality (\ref{eqn: upperbound})  gives, for $T$ sufficiently large,
\begin{align*}
\PRO (M_T (1/z) \geq x) = 1 - \PRO (M_T (1/z) < x) \geq 1 - \PRO (M (1/z) < x) \geq \frac12.
\end{align*}
Hence, we obtain that 
\begin{align}\label{eqn:supLP}
\PRO \left( M_T (T) < x \right)& \geq \frac{\kappa_T (z,0)  V (x)}{4e} - \frac{e}{e-1} V (x) zT c_1(\kappa)  \kappa (1/T,0) .
\end{align}
%%%%%%%%%%%%%%%%%%%%%%%%%%%%%%%%%%%%%%%%%%%%%%%%%%%%%%%%%%%%%%%%%%%%%%%%%%%%%%%%%%%%%%%%%%%%%%%%%%%%%%%%%%%%%%%%%%%%%%%%%%%%%%%%%%%%%%%%%%%%%%%%%%%%%%%%%%%%%%%%%%%%%%%%%%%%%%%%%%%%%%%%%%%%%%%%%
\textit{2nd.\ Step:} 

 In this  step we estimate $\kappa_T$ from below by $\kappa$. Recall that the bivariate Laplace exponent $\kappa$ corresponds to $X$.
 We get, for every $a >0$, that
\begin{align*}
\kappa_T (a,0) &= \kappa (a,0) \cdot \exp \left(\int_0^\infty \left( e^{-t} - e^{-at}  \right)t^{-1} \left( \PRO (Y_T (t) \geq 0) - \PRO (X(t) \geq 0) \right) \dd t \right) .
\end{align*}
 %Since for all $t \geq 0$
 %\begin{align*}
 %\PRO (Y_T (t) \geq 0) \geq \PRO (X(t) \geq S_T (t) , S_T (t) \geq 0) \geq \PRO (X(t) \geq 0).
 %\end{align*} 
 %we obtain for every $a \in (0,1]$
 %\begin{align*}
 %\kappa_T (a, 0) \leq  \kappa (a,0),
 %\end{align*}
 %and  for every $a >1$
 %\begin{align*}
 %\kappa_T (a,0) \geq \kappa (a,0).
 %\end{align*}
 First, we will prove  
 \begin{align*}
 \PRO (Y_T (t) \geq 0) - \PRO (X(t) \geq 0) \leq  o(1),  \quad \text{ as } T \to \infty, 
\end{align*}
for all $t \geq 0$ to obtain finally an estimate for $\kappa$. For this purpose, we distinguish between $0 <t  \leq  \delta(T)^{-1/2} $ and $t >  \delta(T)^{-1/2} $.

For $0 <t  \leq   \delta(T)^{-1/2} $ we have 
\begin{align}\label{eqn:tsmall}
& \PRO (Y_T (t) \geq 0)  - \PRO (X(t) \geq 0)  \notag \\
& \leq \PRO (X (t) \geq -S_T (t), S_T (t) = 0) + \PRO (S_T (t) >0) - \PRO (X(t) \geq 0)  \notag   \\
&= \PRO (X (t) \geq  0) + 1 - \PRO (S_T (t) =0) - \PRO (X(t) \geq 0)  \notag  \\
 &=1 - e^{-t \delta (T)  \nu_+ (1)}  \notag  \\
  &\leq 1 - e^{- \delta(T)^{-1/2} \delta (T) \nu_+ (1)}  \notag  \\
   &\leq C_1 \cdot \delta (T)^{1/2},
\end{align}
where $C_1 = \nu_+ (1)$. 
Now, let $t > \delta(T)^{-1/2}$. It holds that
 \begin{align*}
 \PRO & (Y_T (t) \geq 0) - \PRO (X(t) \geq 0) = \PRO ( - S_T (t) \leq X(t) <0)  \\
  & \leq \PRO \left( - S_T (t) \leq X(t) <0, S_T (t) < c(t) \delta(T)^{\tfrac{1}{2 \alpha} } \right) + \PRO \left( S_T (t) \geq c(t) \delta(T)^{\tfrac{1}{2 \alpha}} \right) \\
 & \leq \PRO \left( - c(t) \delta(T)^{\tfrac{1}{2 \alpha} } \leq X(t) <0  \right) + \PRO \left( S_T (t) \geq c(t) \delta(T)^{\tfrac{1}{2 \alpha}} \right).
\end{align*}
Due to  Stone's local limit theorem (see Theorem 8.4.2 in \cite{binbook} for non-lattice random walks resp.\ Proposition 13 in \cite{DonRiv} for Le\'vy processes) and the fact that the density of any $\alpha$-stable law is bounded  there exists a constant $C_2>0$ such that for all $t> \delta(T)^{-1/2}$ 
\begin{align*}
 \PRO \left(- c(t) \delta(T)^{ \tfrac{1}{2\alpha}}  < X (t) < 0 \right) \leq C_2 \delta(T)^{1/ (3 \alpha) } .
\end{align*}
 Combining this with Remark \ref{rem:ST} and (\ref{eqn:tsmall}) and using $\alpha<1$ here gives uniformly in $t$
\begin{align*}
\PRO & ( - S_T (t) \leq X(t) <0)  \leq  C_2  \delta (T)^{1/6} = o(1), \text{ as } T \to \infty .
\end{align*}
Hence, for $T>1$ sufficiently large we obtain by Frullani's integral for $a \in (0,1]$ that
\begin{align*}
\kappa_T (a,0) &\geq \kappa (a,0) \cdot  \exp \left(  (\ln a) \cdot C_2 \cdot \delta (T)^{1/6} \right).
\end{align*}
%and for $a>1$
%\begin{align*}
%\kappa_T (a,0) &\leq \kappa (a,0) \cdot  \exp   \left(   (\ln a) \cdot C_3 \delta (T)^{1/2} \right).
%\end{align*}
\textit{3rd.\ Step:} 

Inserting this upper bound of $\kappa_T$ in (\ref{eqn:supLP}) leads to
\begin{align*}\PRO \left( M_T (T) < x \right)& \geq \frac{\kappa_T (z,0) V (x)}{4e} - \frac{e}{e-1} V (x) z T c_1(\kappa)  \kappa (1/T,0) \\
 & \geq  \exp   \left(   (\ln  z) \cdot C_2 \delta (T)^{1/6} \right) \frac{\kappa (z,0) V (x)}{4e}  - \frac{e}{e-1} V (x) zT c_1(\kappa)  \kappa (1/T,0) \\
& \geq  z^{   C_2 \delta (T)^{1/6} } \cdot  \frac{\kappa (z,0)  V (x)}{4 e} \\
& \quad \cdot \left( 1 - 4 \frac{e^2}{e-1} \frac{z T c_1(\kappa)  \kappa (1/T,0)}{\kappa (z,0)}   z^{ -  C_2 \delta (T)^{1/6} } \right) .
\end{align*}
Potter's theorem (cf \cite{Bin}, Theorem 1.5.6) implies
\begin{align*}
 \frac{\kappa (1/T,0)}{\kappa (z,0)} \leq \frac{c_2(\kappa)}{(zT)^{(1+\rho)/2}}, \quad \text{ for } z \leq \frac1T.
\end{align*}
Now, set
\begin{align*}
z (T) &:= \left(\frac1T \right)^{\left( 1 + \frac{2 C_2 \delta(T)^{1/6}}{1+\rho} \right)^{-1} } \cdot \left( \frac{e-1}{8e^2 c_1( \kappa) c_2 (\kappa)} \right)^{- \left(\frac{1}{1+\rho} + C_2 \delta(T)^{1/6}\right)^{-1}}\\
& =  T^{- 1 - |o(1)|} \leq T^{- 1},
\end{align*}
for large $T$. Thus,
\begin{align}
4 \frac{e^2}{e-1} \frac{z c_1(\kappa) T \kappa (1/T,0)}{\kappa (z,0)}   z^{ -  C_2 \delta (T)^{1/6}} &\leq 4 \frac{e^2}{e-1}(zT)^{1- (1+\rho)/2}  c_1(\kappa) c_2(\kappa)   z^{ -  C_2 \delta (T)^{1/6}} \notag \\
& \leq \frac12 z T.
\end{align}
Since $zT \leq 1$ Frullani's integral gives
\begin{align*}
\kappa (z,0) &=\exp \left(  \int_0^\infty \left( e^{-t/T} - e^{-tz}  \right)t^{-1}  \PRO (X (t) \geq 0)   \dd t     \right)\\ 
& \qquad \cdot   \exp \left( \int_0^\infty \left( e^{-t} - e^{-t/T}  \right)t^{-1}  \PRO (X (t) \geq 0)   \dd t  \right)\\
 &\geq  \cdot   \exp \left(  \int_0^\infty \left( e^{-t/T} - e^{-tz}  \right)t^{-1} 1   \dd t     \right)\\
& \qquad \exp \left( \int_0^\infty \left( e^{-t} - e^{-t/T}  \right)t^{-1}  \PRO (X (t) \geq 0)   \dd t  \right)\\
&\geq zT \kappa (1/T,0).
\end{align*} 
Hence, we obtain with $zT \geq T^{-|o(1)|} $ that 
\begin{align*}
\PRO \left( M_T (T) < x \right) &\geq   zT \cdot  z^{C_2 \delta (T)^{1/6} } \cdot  \frac{\kappa (1/T,0)  V (x)}{8 e} = T^{-\rho + o(1)},
\end{align*}
where we used that  $\kappa (1/T,0) = T^{-\rho + o(1)} $ by \cite{Rog}.
%Recall that $Y_T$ satisfies  Spitzer's condition with parameter $\rho_Y (T)\in (0,1)$ for every fixed $T>0 $ given by
%\begin{align}\label{eqn: rhoYT}
 %\rho_Y (T) = \frac12 + \frac{1}{\pi \alpha} \arctan (\beta_Y (T) \tan (\pi \alpha/2)).
%\end{align}
%Define $\beta  := \frac{c_1- c_2 }{c_1+c_2}$ and 
%\begin{align*}
% \rho := \frac12 + \frac{1}{\pi \alpha} \arctan (\beta  \tan (\pi \alpha/2)).
%\end{align*}

 \end{proof}

%%%%%%%%%%%%%%%%%%%%%%%%%%%%%%%%%%%%%%%%%%%%%%%%%%%%%%%%%%%%%%%%%%%%%%%%%%%%%%%%%%%%%%%%%%%%%%%%%%%%%%%%%%%%%%%%%%%%%%%%%%%%%%%%%%%%%%%%%%%%%%%%%%%%%%%%%%%%%%%%%%%%%%%%%%%%%%%%%%%%%%%%%%%%%%%%%%%%%%%%%%%%%%%%%%%%%%%%%%%
Next, we show Lemma \ref{lem: constant1}.  Here, we look at the tail behaviour of the first passage time of $Y_T$  defined in (\ref{eqn: charYT1}). Note that this lemma deals with L\'evy processes in discrete time.

\begin{proof}[Proof  of Lemma \ref{lem: constant1}]
Recall that $Y_T = X - S_T $, where $X$ is defined in Theorem \ref{thm: nondecreasing}  and $S_T$ is a subordinator defined in (\ref{eqn: charST1}). Furthermore, note that $(Y_T (n))_{n \in \mathbb{N}}$ with $Y_T$ defined in (\ref{eqn: charYT1}) is a time discrete L\'evy process. The same holds for $(X(n))_{n \in \mathbb{N}}$. By construction, it is clear that 
$(X(n))_{n \in \mathbb{N}}$ satisfies Spitzer's condition with parameter $\rho \in (0,1)$.

The basics of fluctuation theory for the time discrete case are essentially the same as for the time continuous case. In the following we keep the notation for the inverse local time $L^{-1}$ and the ascending ladder process $H$. The bivariate Laplace exponent of $(L^{-1},H)$ is denoted by 
\begin{align*}
  \kappa (a,b) = \exp \left( \sum_{n=0}^\infty \int_{[0,\infty)} (e^{-n} - e^{-an-bx}) n^{-1} \PRO (X (n) \in \dd x)   \right),
\end{align*}
\textit{1st.\ Step:} 

By Proposition 2.4 in \cite{SupLP} we obtain that
\begin{align*}
 \PRO  \left( Y_T (n)  \leq x ,\text{ } \forall n =  1,..., \lfloor T\rfloor   \right) \leq \frac{1}{T(1-e^{-1}) } \sum_{m=0}^\infty e^{-m \cdot \frac{1}{T}}  \PRO  \left( Y_T (n)  \leq x ,\text{ } \forall n =  1,..., m  \right)  .
\end{align*}
By repeating the argument used for the continuous-time case in  \cite{bertoin}, Formula (VI.8), we obtain, for fixed $T>1$, that
\begin{align*}
 \sum_{m=0}^\infty e^{-m \cdot \frac{1}{T}}    \PRO  \left( Y_T (n)  \leq x ,\text{ } \forall n =  1,..., m  \right) \dd t  \leq T \kappa_T(1/T,0) V^{1/T}_T (x), \quad \text{for }x \geq 0.
\end{align*}
By definition we get  $V^{1/T}_T (x) \leq V_T (x)$, for all $T >1$ and $x \geq 0$. Hence,
\begin{align}\label{eqn: upperT}
  \PRO  \left( Y_T (n)  \leq x ,\text{ } \forall n =  1,..., \lfloor T\rfloor   \right)   \leq \min \left(  1, \frac{e}{e-1} \kappa_T (1/T,0) V_T (x) \right).
\end{align}
 Note that \cite{SupLP} proves this statement for the time-continuous case by using the same arguments (see their Theorem 3.1).
 
 The proof is complete as soon as we know that $\kappa_T (1/T,0) \leq T^{-\rho + o(1)}$ and $V_T (x) \leq 2 V(x)$,  for $T >1$ sufficiently large. \\
\textit{2nd.\ Step:}

In this step, we show  the upper bound of $\kappa_T$.  Due to the definition of $Y_T$ and $\kappa_T$ we have for every $T>1$
\begin{align*}
\kappa_T (1/T,0) &= \kappa (1/T,0) \cdot \frac{\kappa_T (1/T,0)}{\kappa (1/T,0)}\\
&= \kappa (1/T,0) \cdot \exp \left( \sum_{n=1}^\infty  (e^{-n/T} - e^{-n}) n^{-1} ( \PRO (X (n) \geq 0) - \PRO (Y_T (n) \geq 0))  \right).
\end{align*}
Now, we will prove
\begin{align*}
 \PRO (X (n) \geq 0) - \PRO (Y_T (n) \geq 0) \to 0, \quad \text{ as $T \to \infty$, uniformly in $n$};
\end{align*}
to obtain finally an estimate for $\kappa$. For this purpose, we distinguish  $n \leq   [\delta(N)^{-1/2}]$ and $n > [\delta(N)^{-1/2}]$. Due to the independence of $S_T$ and $Y_T$ we get, for $T>1$ and  $n \leq [\delta(N)^{-1/2}]$,  
\begin{align}\label{eqn:n1}
& \PRO (X(n) \geq 0) - \PRO (Y_T (n) \geq 0) \notag \\
& \leq   \PRO (Y_T (n) \geq - S_T (n), S_T(n)=0) + \PRO (S_T (n) >0)  -\PRO (Y_T (n) \geq 0) \notag \\
&\leq  \PRO (Y_T (n) \geq 0) \cdot \PRO (S(n)=0) + 1 - \PRO (S_T (n) =0) - \PRO (Y_T (n) \geq 0)\notag  \\
&\leq   1 - e^{- [\delta(N)^{-1/2}] \delta (T)  \nu_+ (1)} \notag \\
   &\leq \nu_+ (n) \cdot \delta (T)^{1/2} .
\end{align}
 Now, let $n >[\delta(N)^{-1/2}]$. It holds that
% Denote by $X^+ $ the L\'evy process which possesses only the positive jumps of $X$ being larger than $c(t) \delta(T)^{\frac{1}{3\alpha}}$. On the other hand, denote by $X^-$ the L\'evy process which possesses the negative jumps of $X$ being smaller than $- c(t) \delta(T)^{\frac{1}{3\alpha}}$.  Note that $X^+$ and $X^-$ are subordinators which have the following Laplace exponents for $\lambda = (c(t) \delta(T)^{\frac{1}{3\alpha}})^{-1}$
% \begin{align*}
% \ER \left( \exp \left(     -    \lambda X^+ (1)    \right) \right) = \exp \left( - \nu_+ (1/ \lambda)   \right)  \text{ resp. } \ER \left( \exp \left(     -    \lambda | X^- (1) |    \right) \right) = \exp \left( - \nu_- (1/ \lambda)   \right).
%  \end{align*} 
%  Furthermore, note that $|X(t)| \leq |X^+ (t)  | + | X^- (t)|$ a.s.
%  Since $\nu_+ \in \RV(-\alpha)$ there exists a function $c \in \RV (1/\alpha)$ such that 
% \begin{align*}
% t \nu_+ (c(t) x) \to  x^{-\alpha}, \quad   x >0  \text{ fixed },  \text{ as } t \to \infty,
% \end{align*} 
% cf.\ \cite{Bin}, Theorem 1.5.12.  
% Hence, Chebyshev's inequality  and (\ref{eqn: abslaplace1})  imply,  for all large $n$ and $T$ sufficiently large, that
\begin{align*}
 \PRO & ( X(n) > 0) - \PRO (Y_T (n) \geq 0) \\
 & \leq \PRO ( 0 < X (n) < S_T (n), S_T (n) < c(n) \delta(T)^{\tfrac{1}{2\alpha} }) + \PRO (S_T (n) \geq c(n) \delta(T)^{\tfrac{1}{2\alpha}})\\ 
 &\leq \PRO \left( 0 < X (n) <  c(n) \delta(T)^{ \tfrac{1}{2 \alpha}}  \right) + \PRO (S_T (n) \geq c(n) \delta(T)^{ \tfrac{1}{2\alpha}}).
% &\leq \PRO \left( 0 < X (n) <  c(n) \delta(T)^{\tfrac16 + \tfrac{1}{3\alpha}}  \right)\\
% & \quad  + \PRO \left( 1 - \exp \left( -S_T (n) \delta (T)^{-\tfrac16 -\tfrac{1}{3\alpha}} / c(n) \right) \geq 1- e^{-1}\right)\\
% &\leq \PRO \left(0 < X (n) <  c(n) \delta(T)^{\tfrac16 + \tfrac{1}{3\alpha}}  \right)\\
% & \quad  +(1- e^{-1})^{-1} \left( 1- \exp \left( - t \delta (T)  \nu_- (c(n) \delta (T)^{\tfrac16 +\tfrac{1}{3\alpha}})  \right) \right)\\
% &\leq \PRO \left( 0 < X (n) <  c(n) \delta(T)^{\tfrac16 + \tfrac{1}{3\alpha}} \right)
%  +(1- e^{-1})^{-1} \left( 1- \exp \left( - C_2 \delta (T)^{\tfrac13}   \right) \right),
\end{align*}
% for  $C_2 >0$ suitably chosen.
Due to  Stone's local limit theorem (see Theorem 8.4.2 in \cite{binbook} for non-lattice random walks), the fact that the density of any $\alpha$-stable law is bounded, and because we are dealing with processes not requiring centering, we obtain that there exists a constant $C_1>0$ such that for $n>[\delta(N)^{-1/2}]$
\begin{align*}
  \PRO \left( 0 < X (n) <  c(n) \delta(T)^{ \tfrac{1}{2\alpha}} \right) \leq C_1 \delta(T)^{1/(3 \alpha)} .
\end{align*}
% Thus, the Chebyshev's inequality leads to the following estimate for $T>1$ sufficiently large and for large $t >1$
% \begin{align*}
% - \PRO &  (Y_T (t) \geq 0) + \PRO (X (t) \geq 0) \\
% & = \PRO  (- S_T (t) \leq Y_T (t) < 0,S_T (t) < c(t) \delta(T)^{\tfrac16 +\tfrac{1}{3\alpha}} ) + \PRO (S_T (t) \geq c(t) \delta(T)^{\tfrac16 +\tfrac{1}{3\alpha}}) \\
% &\leq \PRO ( | Y_T(t)| < c(t) \delta(T)^{\tfrac16 +\tfrac{1}{3\alpha}}) + \PRO (S_T (t) \geq c(t) \delta(T)^{\tfrac16 +\tfrac{1}{3\alpha}})\\ 
% &\leq \PRO \left( \exp \left( -( | Y^+_T (t)| + | Y_T^- (t)|)1/( \delta(T)^{\tfrac{1}{3\alpha}} c(t)) \right) > \exp \left(- \delta(T)^{\tfrac16} \right) \right)\\
% & \quad + \PRO \left( 1 - \exp \left( -S_T (t) \delta (T)^{-\tfrac16 -\tfrac{1}{3\alpha}} / c(t) \right) \geq 1- e^{-1}\right)\\ 
% & \leq \exp \left( \delta(T)^{\tfrac16} - t (1-\delta(T)) \nu_+ ( c(t) \delta(T)^{\frac{1}{3\alpha}} ) - t \nu_- ( c(t) \delta(T)^{\frac{1}{3\alpha}})\right) \\
% & \quad  +(1- e^{-1})^{-1} \left( 1- \exp \left( - t \delta (T)  \nu_+ (c(t) \delta (T)^{\tfrac16 +\tfrac{1}{3\alpha}})  \right) \right)\\
% & \leq \exp \left( \delta(T)^{-\tfrac16 } - C_2 \frac12 \delta(T)^{-\tfrac13} \right)  +(1- e^{-1})^{-1} \left( 1- \exp \left( - C_2 \delta (T)  \delta(T)^{-\tfrac{\alpha + 2}{6}} \right)\right) \\
% & \leq \exp \left( - C_2 \frac14  \delta(T)^{-\tfrac13} \right)  +(1- e^{-1})^{-1} \left( 1- \exp \left( - C_2 \delta (T)^{\tfrac13}   \right) \right).
% \end{align*}
 Combining this with Lemma \ref{lem:ST} and (\ref{eqn:n1}) gives uniformly in $n$
\begin{align*}
 \PRO (X(n) \geq 0)  - \PRO (Y_T(n) \geq 0) \leq   C_2 \delta (T)^{1/6} \to 0, \text{ as } T \to \infty ,
\end{align*}
with a constant $C_2 >0$. Hence, Frullani's integral implies that
\begin{align} \label{eqn: KappaT1}
\kappa_T (1/T,0) \leq \kappa (1/T,0) T^{C_2 \delta(T)^{1/6}}.
\end{align}
\textit{3rd. Step:} 

 Lemma  \ref{lem: VT}  gives for $T>1$ sufficiently large, 
\begin{align}\label{eqn: VT1}
V_T (x) \leq 2  V(x).
\end{align}   
Since $\PRO (X(n) >0) \to \rho$, as $n \to \infty$, it follows from   \cite{Rog} that $\kappa (1/T,0) = T^{-\rho + o(1)}$.
Thus, inserting (\ref{eqn: VT1}) and (\ref{eqn: KappaT1})  in (\ref{eqn: upperT}) leads, for $T>1$ sufficiently large, to
\begin{align*}
 \PRO  \left( Y_T (n)  \leq 1 ,\text{ } \forall n =  1,..., \lfloor T\rfloor   \right) & \leq 
  2 \frac{e}{e-1} \kappa (1/T,0)  T^{-C_2 (\ln \ln T)^{-1/6}} V (x)  \\
  &= T^{-\rho + o(1)},
\end{align*}
as desired.
\end{proof}

%%%%%%%%%%%%%%%%%%%%%%%%%%%%%%%%%%%%%%%%%%%%%%%%%%%%%%%%%%%%%%%%%%%%%%%%%%%%%%%%%%%%%%%%%%%%%%%%%%%%%%%%%%%%%%%%%%%%%%%%%%%%%%%%%%%%%%%%%%%%%%%%%%%%%%%%%%%%%%%%%%%%%%%%%%%%%%%%%%%%%%%%%%%%%%%%%%%%%%%%%%%%%%%%%%%%%%%%%%%%%%%%%%%%%%%%%%%%%

%%%%%%%%%%%%%%%%%%%%%%%%%%%%%%%%%%%%%%%%%%%%%%%%%%%%%%%%%%%%%%%%%%%%%%%%%%%%%%%%%%%%%%%%%%%%%%%%%%%%%%%%%%%%%%%%%%%%%%%%%%%%%%%%%%%%%%%%%%%%%%%%%%%%%%%%%%%%%%%%%%%%%%%%%%%%%%%%%%%%%%%%%%%%%%%%%%%%%%%%%%%%%%%%%%%%%%%%%%%%%%%%%%%%%%%%%%%%%%%%%%%%%%%%%%%%%%%%%%%%%%%%%%%%%%%%%%%%%%%%%%%%%%%%%%%%%%%%%%%%%%%%%%%%%%%%%%%%%%%%%%%%%%%%%%%%%%%%%%%%%%%%%%%%%%%%%%%%
\subsection{First passage time of a time-dependent subordinator}\label{sec:sub}
This section deals with the asymptotic behaviour of the first passage time for a subordinator depending on $T$ over an increasing boundary as $T$ converges to infinity. Lemma \ref{lem: n0N} serves as an auxiliary tool to prove the main result of this section, Lemma \ref{lem: alpha01}.

\begin{lemma}\label{lem: n0N}
Let $\alpha  \in (0,1)$ and $\gamma >0$ with $0 < \gamma \alpha <1$. Let $\eps >0$ such that 
\begin{align*}
 \gamma \alpha - \eps >0  \quad \text{ and } \quad  \gamma \alpha + \eps <1.
\end{align*}
For $N>1$ define $\delta(N)= (\ln \ln N)^{-1} \wedge \tfrac12$ and $N_1 (N)= \lfloor (\ln \ln N)^{4/(1-\gamma \alpha - \eps )} \rfloor$. Furthermore, let $S_N$ be a subordinator with Laplace transform 
 $$\ER (\exp (- \lambda S_N(1))) \leq \exp \left( - \delta (N) \lambda^{\alpha} \ell (\lambda  \delta(N)^{1/\alpha} ) \right), $$
for  $\lambda >0$ sufficiently small and $\ell$  a slowly varying function at zero.

Then, it holds
\begin{align*}
  \PRO \left( S_N (n)  \geq  (n+1)^\gamma , \text{ }\forall  n= N_1 (N) ,...,N  \right) \sim  1 , \quad \text{as } N \rightarrow \infty.
\end{align*}
\end{lemma}

We mention that the last lemma can also be shown for $\alpha\geq 1$.

\begin{proof}
Denote $N_1 := N_1 (N)$ and define $N_0 := N_0 (N) = \lfloor (\ln N)^{4/(1-\gamma \alpha - \eps )} \rfloor $. 

Observe that 
\begin{align*}
   \PRO & \left( S_N (n)  \geq \tfrac12 - (n+1)^\gamma , \text{ }\forall n= N_0,...,N  \right) \\
&= 1 -  \PRO \left( \exists n \in \{ N_0,..., N\} : S_N (n) < (n+1)^{\gamma} \right).
\end{align*}
By Chebyshev's inequality  we obtain, for $N$ sufficiently large, that
\begin{align*}
 \PRO & \left( \exists n \in \{N_1,..., N\} : S_N (n) < (n+1)^{\gamma} \right) \\
& \leq \sum_{n=N_1}^{N_0}  \PRO \left( S_N (n ) < (n+1)^{\gamma}  \right) + \sum_{n=N_0}^N   \PRO \left( S_N (n ) < (n+1)^{\gamma}  \right)\\
 &= \sum_{n=N_1}^{N_0}  \PRO \left((n+1)^{-\gamma} S_N (n) < 1 \right) +  \sum_{n=N_0}^N \PRO \left((n+1)^{-\gamma} S_N (n) < 1 \right)\\
 &= \sum_{n=N_1}^{N_0}  \PRO \left(e^{-(n+1)^{-\gamma} S_N (n)} > e^{-1} \right) + \sum_{n=N_0}^N \PRO \left(e^{-(n+1)^{-\gamma} S_N (n)} > e^{-1} \right)\\
&\leq \sum_{n=N_1}^{N_0}   e^1 \ER \left( e^{- (n+1)^{-\gamma}  S_N (n)   } \right)  + \sum_{n=N_0}^N  e^1 \ER \left( e^{- (n+1)^{-\gamma}  S_N (n)   } \right) \\
 &\leq \sum_{n=N_1}^{N_0}  \exp{ \left( 1  - n (n+1)^{-\gamma \alpha}  \ell( (n+1)^{-\gamma}  \delta(N)^{1/\alpha} ) \delta (N)  \right) } \\
  & \quad +  \sum_{n=N_0}^N \exp{ \left( 1  - n (n+1)^{-\gamma \alpha} \ell((n+1)^{-\gamma} \delta(N)^{1/\alpha}) \delta (N) \right) } .
\end{align*}
Then, Proposition 1.3.6 in \cite{binbook} gives  $\ell (\lambda ) \geq \lambda^{\eps / \gamma}$, for $\lambda >0$ sufficiently small and thus,  for $N$ sufficiently large,
\begin{align*}
 \PRO & \left( \exists n \in \{N_1,..., N\} : S_N (n) < (n+1)^{\gamma} \right) \\
&\leq  \sum_{n=N_1}^{N_0}  \exp{ \left( 1  - n (n+1)^{-\gamma \alpha -\eps }  \delta (N)^{1+ \eps / (\alpha  \gamma) } \right) }  \\
& \quad + \sum_{n=N_0}^N \exp{ \left( 1  - n (n+1)^{-\gamma \alpha - \eps }  \delta (N)^{1+ \eps/ ( \alpha \gamma) } \right) } \\
 &\leq \sum_{n=N_1}^{N_0}  \exp{ \left( 1  - \tfrac12  n^{1-\gamma \alpha -\eps } \delta (N)^{1+ \eps /( \alpha \gamma) } \right) }  +  \sum_{n=N_0}^N \exp{ \left( 1  - \tfrac12 n^{1-\gamma \alpha-\eps } \delta (N)^{1+\eps / ( \alpha \gamma) }  \right) } ,
\end{align*}
where we used in the last last step the fact that $ \gamma \alpha +\eps  <1$ and thus, 
\begin{align*}
 (n+1)^{\gamma \alpha + \eps } \leq n^{\gamma \alpha + \eps } +1 \leq 2 n^{\gamma \alpha+ \eps }, \quad n \geq 1.
\end{align*}
Since $ \gamma \alpha - \eps >0 $ and $ \gamma \alpha +\eps  <1$  we get
\begin{align*}
 \PRO & \left( \exists n \in \{N_1,..., N\} : S_N (n) < (n+1)^{\gamma} \right) \\
 &\leq \exp{ \left( 1 + \ln ( N_0)  -  \tfrac12  N_1^{1-\gamma \alpha-\eps } \delta (N)^{1+ \eps / ( \alpha \gamma) }  \right) }  \\
& \quad + \exp{ \left( 1 + \ln N - \tfrac12  N_0^{1-\gamma \alpha-\eps } \delta (N)^{1+ \eps / (\alpha \gamma) } \right) } \\
 &\leq \exp{ \left( 1 + \left(3/(1-\gamma \alpha)\right) \ln \ln N - \tfrac12 (\ln \ln N)^2  \right) }  + \exp{ \left( 1 + \ln N - (\ln N)^2 \right) } \\
& \leq \exp{ \left( - \ln \ln N \right)}.
\end{align*}
Hence, we obtain finally
\begin{align*}
   \PRO & \left( S_N (n)  \geq \tfrac12 - (n+1)^\gamma , \text{ }\forall n= N_0,...,N  \right) \\
&= 1 -  \PRO \left( \exists n \in \{ N_0,..., N\} : S_N (n) < (n+1)^{\gamma} \right) \\
&\geq 1 - e^{ - \ln \ln N }  \quad \longrightarrow 1, \quad \text{as  } N \rightarrow \infty.
\end{align*}
% Clearly,
% \begin{align*}
%  \PRO \left( S_N (n)  \geq \frac12 - n^\gamma , \text{ } \forall n= N_0,...,N  \right)  \leq 1,
% \end{align*}
% which proves the assertion.
\end{proof}

Now, we proceed with the proof of Lemma \ref{lem: alpha01}.

\begin{proof}[Proof of Lemma \ref{lem: alpha01}]
We start by transforming this problem to the discrete time as follows
\begin{align*}
 \PRO \left( S_N (t)  \geq - \frac12 + t^\gamma , \text{ } 0 \leq t \leq N  \right) \geq \PRO \left( S_N (n)  \geq - \frac12 + (n+1)^\gamma , \forall n =1,..., N  \right),
\end{align*} 
where we used that $S_N$ is nondecreasing.

Since $\alpha, \gamma \alpha \in (0,1)$ there exist constants $\eps_1 \eps_2, \eps_3, \eps_4 >0$  such that $\alpha - \eps_1 >0$, $\alpha + \eps_2 <1$, $\gamma - \eps_3 >0$ and $\alpha \gamma + \eps_4 <1$. Define
\begin{align*}
 \eps := \min \left\{ \frac{\eps_1 \eps_2 \eps_3}{1 + \eps_1} , \eps_4  \right\}.
\end{align*}
Furthermore, define $N_1  := \lfloor (\ln \ln N)^{4/(1-\gamma \alpha- \eps)} \rfloor$.

Note that since $ \gamma \alpha + \eps <1$ we have
\begin{align}\label{eqn:absn+1}
 (n+1)^{\gamma \alpha + \eps} \leq n^{\gamma \alpha+ \eps} +1 \leq 2 n^{\gamma \alpha + \eps}, \quad n \geq 1.
\end{align}
Since $S_N$ is associated (cf.\ \cite{ass} or \cite{AKS}, Lemma 14)  we obtain that
\begin{align}\label{eqn: insert}
 \PRO & \left( S_N (n)  \geq - \frac12 + (n+1)^\gamma , \text{ }  \forall n = 1,..., N \right) \notag \\
& \geq    \PRO  \left( S_N (n)  \geq - \frac12 + (n+1)^\gamma , \text{ } \forall n =1,..., N_1 -1 \right)   \notag  \\
& \quad \cdot  \PRO  \left( S_N (n)  \geq - \frac12 + (n+1)^\gamma , \text{ }  \forall n =N_1,..., N  \right).
\end{align}
Due to Lemma \ref{lem: n0N}  it is left to show that  
\begin{align}\label{eqn:aim}
 \PRO  \left( S_N (n)  \geq - \frac12 + (n+1)^\gamma , \text{ } \forall n =1,..., N_1 -1 \right)  = N^{o(1)}.
\end{align}
Let us show now (\ref{eqn:aim}). \\ % we distinguish $\alpha \in (0,1) $ and $\alpha \in [1,2)$.\\
%\textit{1st.\ Case:}
 Again the  fact that $S_N$ is associated leads to
\begin{align}
  \PRO  \left( S_N (n)  \geq - \frac12 + (n+1)^\gamma , \text{ } \forall n= 1,...,N_1  \right)
& \geq  \PRO  \left( S_N (n)  \geq  (n+1)^\gamma , \text{ } \forall n= 1,...,N_1  \right) \notag \\
& \geq \prod_{n=1}^{N_1 } \PRO  \left( S_N (n)  \geq  (n+1)^\gamma  \right) . \label{eqn:refer20}
\end{align}
 Since $\frac{2 \alpha}{\alpha-1}  < 0$ we obtain $(n+1)^{-\gamma} (\ln \ln N)^{\frac{2 \alpha}{\alpha-1} }\to 0$, as $N \to \infty$, for all $n \geq 1$. Then, applying Chebyshev's inequality, for $N $ sufficiently large, implies (using (\ref{eqn:refer20}))
\begin{align*}
  \PRO & \left( S_N (n)  \geq - \frac12 + (n+1)^\gamma , \text{ } \forall n= 1,...,N_1  \right) \\
& \geq \prod_{n=1}^{N_1 } 1 - \PRO  \left( S_N (n)  <  (n+1)^\gamma  \right) \\
& = \prod_{n=1}^{N_1 } 1 - \PRO  \left( \exp{ \left(  - (\ln \ln N)^{\frac{2}{\alpha-1}}  (n+1)^{-\gamma} S_N (n) \right)}  >  \exp{ \left( -(\ln \ln N)^{\frac{2}{\alpha-1}}  \right)} \right) \\
& \geq  \prod_{n=1}^{N_1 } 1 -  \exp \left( (\ln \ln N)^{\frac{2}{\alpha-1} }  \right. \\
& \qquad \qquad  \left. - \frac{1}{4 \alpha} (\ln \ln N)^{\frac{2 \alpha}{\alpha-1} }  n (n+1)^{ -\gamma \alpha} \ell ((\ln \ln N)^{\frac{2}{\alpha-1}}  (n+1)^{-\gamma}  \delta (N)^{1/\alpha}) \delta (N) \right) . 
\end{align*}
By Proposition 1.3.6 in \cite{binbook} we get $\ell (\lambda) \geq \lambda^{\eps/ \gamma}$ for $\lambda >0$ sufficiently small  and  thus, combining this with (\ref{eqn:absn+1}) gives, for $N $ sufficiently large,
\begin{align*} 
 \PRO & \left( S_N (n)  \geq - \frac12 + (n+1)^\gamma , \text{ } \forall n= 1,...,N_1  \right) \\
& \geq  \prod_{n=1}^{N_1 } 1 -  \exp{ \left( (\ln \ln N)^{\frac{2}{\alpha-1} }  - \tfrac{1}{8 \alpha} (\ln \ln N)^{\frac{2 \alpha + 2 (\eps/ \gamma)}{\alpha-1} }   n^{ 1-\gamma \alpha-\eps } \delta (N)^{1+ \eps / (\alpha \gamma)} \right) }  \\
& \geq  \prod_{n=1}^{N_1 } 1 -  \exp{ \left( (\ln \ln N)^{\frac{2}{\alpha-1}}  - \tfrac{1}{8 \alpha} (\ln \ln N)^{\frac{2 \alpha + 2 (\eps/ \gamma)}{\alpha-1} }  \delta (N)^{1+  \eps / (\alpha \gamma) } \right) }  \\
& \geq  \prod_{n=1}^{N_1 } 1 -  \exp{ \left( (\ln \ln N)^{\frac{2}{\alpha-1}}  - \tfrac{1}{8 \alpha} (\ln \ln N)^{\frac{2 \alpha + 2 (\eps/ \gamma) - \alpha +1   + \eps /  (\alpha \gamma)  - \eps /   \gamma }{\alpha-1}} \right) }
\end{align*}
Recall that  $\eps \leq  \frac{\eps_1 \eps_2 \eps_3}{1 + \eps_1}$. Thus,
\begin{align*}
 2 \alpha + 2 (\eps/ \gamma) - \alpha +1   + \eps /  (\alpha \gamma)  - \eps /   \gamma & =   \alpha + 1 + \eps \left( \frac{1}{\gamma} + \frac{1}{\gamma \alpha} \right) \\
& \leq  \alpha + 1 + \eps \frac{1 + \eps_1}{\eps_1 \eps_3} 
 \leq  \alpha + 1 + \eps_2 < 2.
\end{align*}
Thus, we have, for $N$ sufficiently large,
\begin{align*}
 (\ln \ln N)^{\frac{2 }{\alpha-1}}- \tfrac{1}{8 \alpha} (\ln \ln N)^{\frac{2 \alpha + 2 (\eps/ \gamma) - \alpha +1   + \eps /  (\alpha \gamma)  - \eps /   \gamma }{\alpha-1}} 
& \leq - \tfrac{1}{10 \alpha} (\ln \ln N)^{ \frac{\alpha+1 + \eps_2}{\alpha-1}  }  <0. 
\end{align*}
Therefore, finally proving (\ref{eqn:aim}) we get that, for $N$ sufficiently large,
  \begin{align*}
    \PRO   \left( S_N (n)  \geq - \frac12 + (n+1)^\gamma , \text{ } n= 1,...,N_1  \right) 
  & \geq \frac{1}{2} \prod_{n=1}^{N_1 } 1 -  \exp{ \left(  - \frac{1}{10 \alpha} (\ln \ln N)^{ \frac{\alpha+1+ \eps_2}{\alpha-1} }  \right) }  \\
%   & \geq  \prod_{n=1}^{N_1 } 1 -  e^{ - \frac14 (\ln \ln N)^{ - \frac{\alpha+1 + \eps_3}{1-\alpha} } }  \\
   & \geq  \frac{1}{2}  \left( \frac{1}{10 \alpha} \right)^{N_1} \left( (\ln \ln N)^{- \frac{\alpha+1 + \eps_2}{1-\alpha}  } \right)^{N_1} \\
& = N^{o(1)}.
  \end{align*}
\end{proof}

% \subsection{Properties of slowly varying functions}
% For the sake of completeness  we present a property of a slowly varying function which we used several times in the last sections. 
% \begin{lemma}\label{lem: propell}
%  Let $\ell:(0,\infty) \to (0,\infty)$ be a slowly varying function at infinity. Then, there exists a $T$ sufficiently large such that for all $t \geq T$ we have $\ell(t)\geq t^{-\delta}$ with $\delta>0$.
% % \item For all $t < T$ holds $\ell(t) \geq T^{-\delta}$ with $\delta >0$.  
% \end{lemma}
% \begin{proof}
%  Suppose the assertion is wrong. Then, there exists a sequence $t_i \geq T$ with $t_i \nearrow \infty$  such that $t_i \leq t_{i+1}$ and $\ell (t_i) < t_i^{-\delta}$. If there are only finitely many $t_i$ and $\ell (t_i) < t_i^{-\delta} $, then choose $T = \sup_{i \in \mathbb{N}} t_i+1 < \infty$ and (i) is proved. Proposition 1.3.6 in \cite{binbook} yields 
% \begin{align*}
%   \frac{\ell (2t_i)}{\ell(t_i)} \geq  \ell (2 t_i) t_i^{\delta} \longrightarrow \infty, \quad \text{as } t_i \to \infty, 
% \end{align*}
%   which is a contraction to the slowly varying property of $\ell$.

% Next, we show (ii).  Suppose the assertion is wrong. Then, there exists a sequence $\{t_i \geq T, i \in \mathbb{N} \}$ with $\ell (t_i) < T^{-\delta}$. It holds that $t_i \to \infty$. Otherwise since $\ell >0$ there exists a constant $c>0$ such that for all $x \in [t_1,\max_{i \in  \mathbf{N}} t_i]$ with $\ell (x) > c$. Choose $T= c^\delta$, but this is a contraction to the assertion. Hence, $t_i \to \infty$. But, then this is a contraction to (i).   
% \end{proof}

\noindent {\bf Acknowledgement:} Frank Aurzada and Tanja Kramm were supported by the DFG Emmy Noether programme. We are grateful to Mikhail Lifshits for valuable discussions on the subject of the article. We would like to thank the two referees for their detailed comments, which improved the exposition of the paper significantly and corrected earlier mistakes.

\bibliographystyle{abbrv}

\begin{thebibliography}{10}

\bibitem{AK}
F.~Aurzada and T.~Kramm.
\newblock First exit of {B}rownian motion from a one-sided moving boundary.
\newblock In: {\em High Dimensional Probability VI: The Banff Volume.} Progress in Probability 66, 215--219, 2013.

\bibitem{AKS}
F.~Aurzada, T.~Kramm, and M.~Savov.
\newblock First passage times of {L}\'evy processes over a one-sided moving
  boundary.
\newblock {\em Markov Processes and Related Fields}, to appear, http://arxiv.org/abs/1201.1118.

\bibitem{AurSim}
F.~Aurzada and T.~Simon.
\newblock Persistence probabilities \& exponents.
\newblock {\em L\'evy matters}, Springer, to appear, http://arxiv.org/abs/1203.6554.

\bibitem{Bal}
A.~Baltr{\=u}nas.
\newblock Some asymptotic results for transient random walks with applications
  to insurance risk.
\newblock {\em J. Appl. Probab.}, 38(1):108--121, 2001.

\bibitem{bertoin}
J.~Bertoin.
\newblock {\em {L}\'evy processes}.
\newblock Cambridge Univ. Press, Cambridge, 1996.

\bibitem{BerDon2}
J.~Bertoin and R.~A. Doney.
\newblock Some asymptotic results for transient random walks.
\newblock {\em Adv. in Appl. Probab.}, 28(1):207--226, 1996.

\bibitem{BerDon}
J.~Bertoin and R.~A. Doney.
\newblock Spitzer's condition for random walks and {L}\'evy porcesses.
\newblock {\em Ann. Iinst. Henri Poincar\'e}, 33:167--178, 1997.

\bibitem{Bin}
N.~H. Bingham.
\newblock Limit theorems in fluctuation theory.
\newblock {\em Advances in Appl. Probability}, 5:554--569, 1973.

\bibitem{binbook}
N.~H. Bingham, C.~M. Goldie, and J.~L. Teugels.
\newblock {\em Regular variation}, volume~27 of {\em Encyclopedia of
  Mathematics and its Applications}.
\newblock Cambridge University Press, Cambridge, 1989.

\bibitem{Bor1}
A.~A. Borovkov.
\newblock On the asymptotic behavior of the distributions of first-passage
  times {I}.
\newblock {\em Math. Notes}, 75:23--37, 2004.

\bibitem{Bor2}
A.~A. Borovkov.
\newblock On the asymptotic behavior of the distributions of first-passage
  times {II}.
\newblock {\em Math. Notes}, 75:322--330, 2004.

\bibitem{BMS13}
A.~J. Bray, S.~N. Majumdar, and G.~Schehr.
\newblock Persistence and first-passage properties in non-equilibrium systems.
\newblock {\em Advances in Physics}, 62(3):225--361, 2013.

\bibitem{Brei}
L.~Breiman.
\newblock {\em Probability}, volume~7 of {\em Classics in Applied Mathematics}.
\newblock Society for Industrial and Applied Mathematics (SIAM), Philadelphia,
  PA, 1992.
% \newblock Corrected reprint of the 1968 original.

\bibitem{DDS08}
D.~Denisov, A.~B. Dieker, and V.~Shneer.
\newblock Large deviations for random walks under subexponentiality: the
  big-jump domain.
\newblock {\em Ann. Probab.}, 36(5):1946--1991, 2008.

\bibitem{DenShn}
D.~Denisov and V.~Shneer.
\newblock Asymptotics for first-passage times of {L}\'evy processes and random
  walks.
\newblock {\em J.\ Appl.\ Probab.}, 50(1):64--84, 2013.

\bibitem{denisovwachtel}
D.~Denisov and V.~Wachtel.
\newblock Exact asymptotics for the instant of crossing a curve boundary by an asymptotically stable random walk.
\newblock Preprint, arXiv:1403.5918.

\bibitem{DonRiv}
R.~Doney and V.~Rivero.
\newblock Asymptotic behaviour of first passage time distributions for {L}\'evy
  processes.
\newblock {\em Probab. Theory Relat. Fields} 157(1--2):1--45, 2013. 

\bibitem{Don2}
R.~A. Doney.
\newblock On the asymptotic behaviour of first passage times for transient
  random walk.
\newblock {\em Probab. Theory Related Fields}, 81(2):239--246, 1989.

\bibitem{DonMal}
R.~A. Doney and R.~A. Maller.
\newblock Moments of passage times for {L}\'evy processes.
\newblock {\em Ann. Inst. H. Poincar\'e Probab. Statist.}, 40(3):279--297,
  2004.

\bibitem{DonMal2}
R.~A. Doney and R.~A. Maller.
\newblock Passage times of random walks and {L}\'evy processes across power law
  boundaries.
\newblock {\em Probab. Theory Related Fields}, 133(1):57--70, 2005.

\bibitem{ass}
J.~D. Esary, F.~Proschan, and D.~W. Walkup.
\newblock Association of random variables, with applications.
\newblock {\em Ann. Math. Statist.}, 38:1466--1474, 1967.

\bibitem{feller}
W.~Feller.
\newblock {\em An introduction to probability theory and its applications.
  {V}ol. {II}.}
\newblock Second edition. John Wiley \& Sons Inc., New York, 1971.

\bibitem{Gae}
J.~G{\"a}rtner.
\newblock Location of wave fronts for the multidimensional {KPP} equation and
  {B}rownian first exit densities.
\newblock {\em Math. Nachr.}, 105:317--351, 1982.

\bibitem{GreNov}
P.~E. Greenwood and A.~A. Novikov.
\newblock One-sided boundary crossing for processes with independent
  increments.
\newblock {\em Teor. Veroyatnost. i Primenen.}, 31(2):266--277, 1986.

\bibitem{GrifMal}
P.~S. Griffin and R.~A. Maller.
\newblock Small and large time stability of the time taken for a {L}\'evy
  process to cross curved boundaries.
\newblock {\em Ann. Inst. H. Poincar\'e Probab. Statist.} 49(1):208--235, 2013.

\bibitem{Gut}
A.~Gut.
\newblock On the moments and limit distributions of some first passage times.
\newblock {\em Ann. Probability}, 2:277--308, 1974.

\bibitem{jenler}
C.~Jennen and H.~R. Lerche.
\newblock First exit densities of {B}rownian motion through one-sided moving
  boundaries.
\newblock {\em Z. Wahrsch. Verw. Gebiete}, 55(2):133--148, 1981.

\bibitem{SupLP}
M.~Kwasnicki, J.~Malecki, and M.~Ryznar.
\newblock Suprema of {L}\'evy processes.
\newblock {\em Ann. Probab.} 41(3B): 2047--2065, 2013.

\bibitem{kyp}
A.~Kyprianou.
\newblock {\em Introductory Lectures on Fluctuations of L\'evy Processes with Applications.}
\newblock Springer, Berlin 2006.


\bibitem{mogpec}
A.~A. Mogul'skii and E.~A. Pecherskii.
\newblock The time of first entry into a region with curved boundary.
\newblock {\em Sib. Math. Zh.}, 19:824--841, 1978.

\bibitem{Novref}
A.~A.\ Novikov.
\newblock A martingale approach to first passage problems and a new condition for Wald's identity.
\newblock {\it Stochastic differential systems (Visegr\'ad, 1980)}, pp. 146--156,
\newblock Lecture Notes in Control and Information Sci., 36, {\it Springer, Berlin-New York}, 1981.

\bibitem{Nov2}
A.~A. Novikov.
\newblock The martingale approach in problems on the time of the first crossing
  of nonlinear boundaries.
\newblock {\em Trudy Mat. Inst. Steklov.}, 158:130--152, 230, 1981.
\newblock Analytic number theory, mathematical analysis and their applications.

\bibitem{Nov}
A.~A. Novikov.
\newblock On estimates and asymptotic behavior of nonexit probabilities of
  {W}iener process to a moving boundary.
\newblock {\em Math. USSR Sbornik}, 38:495--505, 1981.

\bibitem{Novdis}
A.~A. Novikov.
\newblock The crossing time of a one-sided non-linear boundary by sums of
  independent random variables.
\newblock {\em Theory Probab. Appl.}, 27:688--702, 1982.

\bibitem{Novneu}
A.~A. Novikov.
\newblock Martingales, a {T}auberian theorem, and strategies for games of
  chance.
\newblock {\em Teor. Veroyatnost. i Primenen.}, 41(4):810--826, 1996.

\bibitem{PevShi}
G.~Pe{\v{s}}kir and A.~N. Shiryaev.
\newblock On the {B}rownian first-passage time over a one-sided stochastic
  boundary.
\newblock {\em Teor. Veroyatnost. i Primenen.}, 42(3):591--602, 1997.

\bibitem{Rog}
B.~A. Rogozin.
\newblock Distribution of the first ladder moment and height, and fluctuations
  of a random walk.
\newblock {\em Teor. Verojatnost. i Primenen.}, 16:539--613, 1971.

\bibitem{Rot}
V.~Rotar'.
\newblock On the moments of the value and the time of the first passage over a
  curvilinear boundary.
\newblock {\em Theory Prob. Applications}, 12:690--691, 1967.

\bibitem{Sal}
P.~Salminen.
\newblock On the first hitting time and the last exit time for a {B}rownian
  motion to/from a moving boundary.
\newblock {\em Adv. in Appl. Probab.}, 20(2):411--426, 1988.

% \bibitem{SamTaq94}
% G.~Samorodnitsky and M.~S. Taqqu.
% \newblock L\'evy measures of infinitely divisible random vectors and {S}lepian
%   inequalities.
% \newblock {\em Ann. Probab.}, 22(4):1930--1956, 1994.

\bibitem{SamTaq}
G.~Samorodnitsky and M.~S. Taqqu.
\newblock {\em Stable non-{G}aussian random processes}.
\newblock Stochastic Modeling. Chapman \& Hall, New York, 1994.
\newblock Stochastic models with infinite variance.

\bibitem{sato}
K.~Sato.
\newblock {\em L\'evy processes and infinitely divisible distributions},
  volume~68 of {\em Cambridge Studies in Advanced Mathematics}.
\newblock Cambridge University Press, Cambridge, 1999.

\bibitem{Uch}
K.~Uchiyama.
\newblock Brownian first exit from and sojourn over one-sided moving boundary
  and application.
\newblock {\em Z. Wahrsch. Verw. Gebiete}, 54(1):75--116, 1980.

\bibitem{VO}
Z.~Vondra{\v{c}}ek.
\newblock Asymptotics of first-passage time over a one-sided stochastic
  boundary.
\newblock {\em J. Theoret. Probab.}, 13(1):279--309, 2000.

\bibitem{Zolo}
V.~M. Zolotarev.
\newblock {\em One-dimensional stable distributions}, volume~65 of {\em
  Translations of Mathematical Monographs}.
\newblock American Mathematical Society, 1986.

\end{thebibliography}

\end{document}